\documentclass[onefignum,onetabnum]{siamart171218}
\usepackage{lipsum}
\usepackage{amsfonts}
\usepackage{graphicx}
\usepackage{epstopdf}
\usepackage{algorithmic}
\usepackage{mathrsfs}
\ifpdf
  \DeclareGraphicsExtensions{.eps,.pdf,.png,.jpg}
\else
  \DeclareGraphicsExtensions{.eps}
\fi

\newcommand{\ud}{\,\mathrm{d}}

\newsiamremark{remark}{Remark}
\newsiamremark{hypothesis}{Hypothesis}
\crefname{hypothesis}{Hypothesis}{Hypotheses}
\newsiamthm{claim}{Claim}
\newtheorem{scheme}{Scheme}
\theoremstyle{plain}
\newtheorem{exam}{Example}

\headers{PDFB splitting for cross diffusion}{Y.Deng, C.Wei}

\title{A Primal-dual Forward-backward Splitting Method for Cross-diffusion Gradient Flows with General Mobility Matrices\thanks{Submitted to the editors DATE.
\funding{CW is supported by the National Natural Science Foundation of China under grants 12371392 and 12431015.}}}

\author{Yunhong Deng\thanks{School of Mathematics, University of Minnesota,
Minneapolis, MN 55455, USA, and School of Mathematical Sciences, University of Electronic Science and Technology of China, Chengdu, Sichuan 611731, China (\email{deng0335@umn.edu}).}
\and Chaozhen Wei\thanks{School of Mathematical Sciences, University of Electronic Science and Technology of China, Chengdu, Sichuan 611731, China (\email{cwei4@uestc.edu.cn}).}}

\usepackage{amsopn}

\makeatletter
\newcommand*{\addFileDependency}[1]{% argument=file name and extension
  \typeout{(#1)}% latexmk will find this if $recorder=0 (however, in that case, it will ignore #1 if it is a .aux or .pdf file etc and it exists! if it doesn't exist, it will appear in the list of dependents regardless)
  \@addtofilelist{#1}% if you want it to appear in \listfiles, not really necessary and latexmk doesn't use this
  \IfFileExists{#1}{}{\typeout{No file #1.}}% latexmk will find this message if #1 doesn't exist (yet)
}
\makeatother

\ifpdf
\hypersetup{
    pdftitle = {A Primal-dual Forward-backward Splitting Method for Cross-diffusion Gradient Flows with General Mobility Matrices},
    pdfauthor={Yunhong Deng, Chaozhen Wei}
}
\fi

\begin{document}

\maketitle

% REQUIRED
\begin{abstract}
In this work, we construct a primal-dual forward-backward (PDFB) splitting method for computing a class of cross-diffusion systems that can be formulated as gradient flows under transport distances induced by matrix mobilities. By leveraging their gradient flow structure, we use minimizing movements as the variational formulation and compute these cross-diffusion systems by solving the minimizing movements as optimization problems at the fully discrete level. Our strategy to solve the optimization problems is the PDFB splitting method outlined in our previous work \cite{PDFB2024}. The efficiency of the proposed PDFB splitting method is demonstrated on several challenging cross-diffusion equations from the literature. 
\end{abstract}

% REQUIRED

% REQUIRED
\section{Introduction}
Cross-diffusion systems arise in a variety of applications involving multi-species or multi-component systems including, to name just a few, chemistry \cite{GiovangigliVincent2004, DesvillettesLaurent2007}, biology \cite{SKT, david2024degenerate}, thin film dynamics \cite{Garcke2006, Jachalski2012, Thiele2012}, and multi-component flows \cite{Elliott2000, Walkington, murphy2019control}. In this work, we consider a general cross-diffusion system of the following form \cite{Sun2018, Carrillo2022},
\begin{equation}\label{1-1}
    \partial_{t}\boldsymbol{\mu} = \nabla \cdot \bigg[M(\boldsymbol{\mu}) \nabla \frac{\delta \mathcal{E}[\boldsymbol{\mu}]}{\delta \boldsymbol{\mu}}\bigg] \ \text{ in } \Omega,
\end{equation}
with a non-flux boundary condition,
\begin{equation*}
    \bigg[M(\boldsymbol{\mu}) \nabla \frac{\delta \mathcal{E}[\boldsymbol{\mu}]}{\delta \boldsymbol{\mu}}\bigg]_{\alpha} \cdot \textbf{n} = 0 \text{ on } \partial\Omega,\quad \forall 1 \le \alpha \le n,
\end{equation*}
where $\boldsymbol{\mu} = [\mu_{\alpha}]_{\alpha = 1}^{n}\in \mathbb{R}^{n \times 1}$ is a vector field representing the population densities or concentrations of $n$ type of species, $\Omega\subset\mathbb{R}^d$ is a regular bounded domain, and $\textbf{n}$ is the unit outer normal vector of $\Omega$. In particular, we consider the concentration-dependent mobility for general cross-diffusion systems
\begin{equation*}
    M(\boldsymbol{\mu}) = [M_{\alpha \beta}(\boldsymbol{\mu})]_{\alpha, \beta = 1}^{n} \in \mathbb{R}^{n \times n},
\end{equation*}
which is a (symmetric) positive-definite matrix field. By denoting $[\,\cdot\,]_{\alpha}$ the $\alpha$-th row of the matrix, the gradient operator $\nabla$ acting on a vector field $\boldsymbol{\mu} = [\mu_{\alpha}]_{\alpha = 1}^{n}\in\mathbb{R}^{n \times 1}$, and the divergence operator $\nabla \cdot$ acting on a matrix field $\textbf{u} = [u_{\alpha, \beta}]_{\alpha, \beta = 1}^{n, d}\in\mathbb{R}^{n \times d}$, are hereinafter defined row-wise operations, i.e.,
\begin{equation*}
    \nabla \boldsymbol{\mu}(x) = \left[\begin{array}{ccc}
        \displaystyle\frac{\partial \mu_{1}(x)}{\partial x_{1}} &\cdots & \displaystyle\frac{\partial \mu_{1}(x)}{\partial x_{d}}\\
        \vdots & & \vdots\\
        \displaystyle\frac{\partial \mu_{n}(x)}{\partial x_{1}} &\cdots & \displaystyle\frac{\partial \mu_{n}(x)}{\partial x_{d}}
    \end{array}\right],\quad
    \nabla \cdot \textbf{u}(x) =
    \sum_{\beta = 1}^{d} \left[\begin{array}{c}
        \displaystyle\frac{\partial u_{1, \beta}(x)}{\partial x_{\beta}}\\
        \vdots\\
        \displaystyle\frac{\partial u_{n, \beta}(x)}{\partial x_{\beta}}
    \end{array}\right],
\end{equation*}
where $x = (x_{1}, \cdots, x_{d})$. Consider the free energy functional $\mathcal{E}[\boldsymbol{\mu}]$ in a general form given by
\begin{equation}\label{energy}
    \mathcal{E}[\boldsymbol{\mu}] = \int_{\Omega} \,U(\boldsymbol{\mu}(x)) \,\mathrm{d}x + \int_{\Omega} \,V(x) \cdot \boldsymbol{\mu}(x)\,\mathrm{d}x + \sum_{\alpha = 1}^{n}\int_{\Omega}\frac{\varepsilon^{2}}{2}|\nabla \mu_{\alpha}(x)|^{2} \,\mathrm{d}x,
\end{equation}
where the Fr\'echet derivative of $\mathcal{E}[\boldsymbol{\mu}]$ is given by
\begin{equation}\label{energy_first}
    \frac{\delta \mathcal{E}[\boldsymbol{\mu}]}{\delta \boldsymbol{\mu}} = \bigg[\frac{\delta \mathcal{E}[\boldsymbol{\mu}]}{\delta \mu_{\alpha}}\bigg]_{\alpha = 1}^{n} = \bigg[\frac{\partial U(\boldsymbol{\mu}(x))}{\partial \mu_{\alpha}} + V_{\alpha}(x) - \varepsilon^{2} \Delta \mu_{\alpha}(x)\bigg]_{\alpha = 1}^{n},
\end{equation}
where $U : \text{dom}(U) \subseteq \mathbb{R}^{n \times 1} \to \mathbb{R}$ represents the internal energy, and $V = [V_{\alpha}]_{\alpha = 1}^{n}: \Omega \to \mathbb{R}^{n \times 1}$ represents the confinement potential. The last term in \eqref{energy} represents the Dirichlet energy usually considered in phase separation problems, and when it is considered, we impose an additional Neumann boundary condition $\nabla \mu_{\alpha} \cdot \textbf{n} = 0$ on $\partial \Omega$ for all $\alpha$.

Efficient numerical methods for cross-diffusion systems have become a popular topic in modern computational physics. To reflect the fundamental properties of the physical system~\eqref{1-1}, it is important to construct numerical methods that preserve key structural properties, such as energy dissipation
\begin{equation*}
    \mathcal{E}[\boldsymbol{\mu}(t + \Delta t)] \le \mathcal{E}[\boldsymbol{\mu}(t)] \text{ for all } \Delta t > 0,
\end{equation*}
and bounded preservation $$0\preceq_{\mathcal{S}_{+}}M(\boldsymbol{\mu}(t)),$$ along the flow~\eqref{1-1}. For example, Bailo, Carrillo, and Hu \cite{Carrillo2022} construct a finite volume method via combining carefully chosen numerical fluxes and techniques for hyperbolic conservation laws, which preserves the natural bound of the numerical solution, and guarantees the energy dissipation at the fully discrete level. Sun, Carrillo, and Shu \cite{Sun2018} proposed a discontinuous Galerkin method which preserves the energy dissipation structure by choosing appropriate numerical fluxes, which also ensures the positivity of the numerical solution via a suitable scaling limiter.

Our approach to computing~\eqref{1-1} is based on the gradient flow framework from optimal transport theory developed over the past few decades \cite{Jordan1996, Otto2001, Carrillo2003, Ambrosio2005, Villani2003}. This framework allows us to interpret equations of the form~\eqref{1-1} as gradient flows and to characterize their dynamics via the minimizing movement scheme \cite{Jordan1996}. For general cross-diffusion systems~\eqref{1-1}, the gradient flow structure was established by Mielke \cite{MielkeAlexander2011}, Liero and Mielke \cite{LieroMatthias2013Gsag}, who formulated~\eqref{1-1} as a gradient flow with respect to a transport metric induced by the matrix mobility. This transport metric was further studied in \cite{Savare2016}. Here we summarize their construction in the following scheme, which generalizes several minimizing movement schemes \cite{Carrillo2019, Carrillo2023} originating from the seminal work \cite{Jordan1996} developed by Jordan, Kinderlehrer, and Otto.

\begin{definition}[Time-discrete scheme]\rm
Given $\tau > 0$ and an initial value $\boldsymbol{\mu}_{0}$, we define the time-discrete solution $\boldsymbol{\mu}_{\tau}$ by $\boldsymbol{\mu}_{\tau}(t) = \boldsymbol{\mu}^{k}$ when $(k - 1) \tau < t \le k \tau$, where $\boldsymbol{\mu}^{0} = \boldsymbol{\mu}_{0}$ and $(\boldsymbol{\mu}^{k})$ is defined iteratively via a sequence of optimization problems:
\begin{equation}\label{JKOMM}
    \boldsymbol{\mu}^{k + 1} = \arg\min_{\boldsymbol{\mu}} \tau \mathcal{E}[\boldsymbol{\mu}] + \frac{1}{2}D(\boldsymbol{\mu}, \boldsymbol{\mu}^{k})^{2},
\end{equation}
where $D(\boldsymbol{\mu}, \boldsymbol{\nu})$ is a transport distance defined as follows,
\begin{equation}\label{SSD}
    \begin{aligned}
        D(\boldsymbol{\mu}, \boldsymbol{\nu}) = \min_{\boldsymbol{\omega},\textbf{u}} \Bigg\{\int_{0}^{1}\int_{\Omega}\, \textbf{u} : M(\boldsymbol{\omega})\textbf{u} &\,\mathrm{d}x \mathrm{d}t : \boldsymbol{\omega}(0) = \boldsymbol{\mu} \text{ and } \boldsymbol{\omega}(1) = \boldsymbol{\nu}\\
        &\partial_{t}\boldsymbol{\omega} + \nabla\cdot \big[M(\boldsymbol{\omega})\textbf{u}\big] = 0 \text{ in } \Omega \times (0, 1)\Bigg\}^{\frac{1}{2}}.
    \end{aligned}
\end{equation}
where $\textbf{u} : \textbf{v} = \sum_{\alpha, \beta = 1}^{n, d} u_{\alpha, \beta} v_{\alpha, \beta}$ is the Frobenius inner product between two matrices $\textbf{u} = [u_{\alpha, \beta}]_{\alpha, \beta = 1}^{n, d}$ and $\textbf{v} = [v_{\alpha, \beta}]_{\alpha, \beta = 1}^{n, d}$.
\end{definition}

We will also provide a derivation of this minimizing movement scheme in Section~\ref{sec:1}. Substantial progress has been made in the rigorous analysis of the scheme~\eqref{JKOMM}. For $n = 1$, the existence of a unique solution $\boldsymbol{\mu}_{\tau}$ and its convergence as $\tau \to 0$ to the corresponding gradient flow have been established for Wasserstein gradient flows \cite{Jordan1996, Ambrosio2005, Blanchet2008} and for general concave mobilities \cite{Lisini2012}. For coupled Wasserstein gradient flows where, $M(\boldsymbol{\mu}) = \text{diag}(\boldsymbol{\mu})$, the minimizing movement scheme has been rigorously analyzed \cite{kim2018nonlinear, ducasse2023cross}. For general cross-diffusion gradient flows, Zinsl and Matthes \cite{ZinslJonathan2015Tdag} proved the existence of a unique solution $\boldsymbol{\mu}_{\tau}$ and its convergence to~\eqref{1-1} in the one-dimensional case ($d = 1$).

In this work, we compute~\eqref{1-1} by solving minimizing movement scheme~\eqref{JKOMM} at a discrete level as an optimization problem. There are several advantages to construct numerical solvers based on~\eqref{JKOMM}. For example, it guarantees energy dissipation and bounded preservation at the time discrete level (see Theorem~\ref{structure} below). Several progresses in numerical methods have been made in this topic in recent years, for example, Wasserstein gradient flows \cite{Benamou2016, Carrillo2019}, Wasserstein-like gradient flows with nonlinear mobilities \cite{Carrillo2023, PDFB2024}, and reaction-diffusion systems with diagonal matrix mobilities \cite{Li2023, fu2024generalized}.

This general form~\eqref{JKOMM} where $M(\boldsymbol{\mu})$ is non-diagonal can be non-trivial to compute. Current computational methods for cross-diffusion systems based on this framework are mainly restricted to the case where $M(\boldsymbol{\mu})$ is diagonal \cite{cances2019simulation, cances2022convergent, Li2023, fu2024generalized}. One technical difficulty in the general form~\eqref{JKOMM} is the fact that the mobility of each species are coupled. Thanks to our previous work \cite{PDFB2024}, we can overcome this technical difficulty by using the primal-dual forward-backward (PDFB) splitting method \cite{PDFB2024}, where one can decouple the mobility matrix through auxiliary dual variables.

The rest of the paper is organized as follows. In Section~\ref{sec:1}, we review the gradient flow structure of the cross-diffusion system~\eqref{1-1} and derive the corresponding minimizing movement scheme. In Section~\ref{sec:2}, we construct a staggered grid approximation of the variational problem~\eqref{JKOMM}, and derive a saddle point form that can be solved by operator splitting methods. In Section~\ref{sec:3}, we discuss the PDFB splitting method \cite{PDFB2024} to solve the saddle point form via optimization solvers. In Section~\ref{sec:4}, we present the numerical results of the PDFB splitting method on challenging numerical examples from the literature. Section~\ref{sec:5} is our conclusion and future directions.

\section{Gradient flow structure}\label{sec:1}
In this section, we define the gradient flow structure of~\eqref{1-1} via an energy dissipation inequality (EDI) form of gradient flows (see, for example, \cite{Mass2022, Fagioli2022, gao2023homogenization}), and obtain the minimizing movement scheme via AGS theory of gradient flows \cite{Ambrosio2005}.

\subsection{Energy dissipation inequality}
Equation~\eqref{1-1} defines a dissipative system where the free energy $\mathcal{E}[\boldsymbol{\mu}]$ acts as a Lyapunov functional. Moreover, let $t \mapsto \boldsymbol{\mu}(t)$ be a solution of~\eqref{1-1}, we have the following form of energy dissipation:
\begin{equation}\label{ed}
    \begin{aligned}
        \frac{\mathrm{d}}{\mathrm{d}t}\mathcal{E}[\boldsymbol{\mu}(t)] = -A\bigg(\boldsymbol{\mu}(t), -\frac{\delta \mathcal{E}[\boldsymbol{\mu}(t)]}{\delta \boldsymbol{\mu}}\bigg) \le 0,
    \end{aligned}
\end{equation}
where $A$ is the entropy dissipation functional (for $\boldsymbol{\varphi} : \Omega \to \mathbb{R}^{n \times 1}$) given by
\begin{equation*}
    A(\boldsymbol{\mu}, \boldsymbol{\varphi}) = \int_{\Omega}\, \nabla \boldsymbol{\varphi} : M(\boldsymbol{\mu})\nabla \boldsymbol{\varphi}\,\mathrm{d}x.
\end{equation*}
The form of dissipation in~\eqref{ed} gives us the following EDI form of the cross diffusion gradient flow~\eqref{1-1}.
\begin{lemma}[EDI form of gradient flow]\label{EDI-lemma}Let $t \mapsto \boldsymbol{\mu}(t)$ be the solution of~\eqref{1-1} with initial value $\boldsymbol{\mu}_{0}$, it can be equivalently characterized by
\begin{equation}\label{EDI}
    \mathcal{E}[\boldsymbol{\mu}(T)] + \int_{0}^{T} \frac{1}{2}A^{\ast}\big(\boldsymbol{\mu}(t), \partial_{t}\boldsymbol{\mu}(t)\big) + \frac{1}{2}A\bigg(\boldsymbol{\mu}(t), -\frac{\delta \mathcal{E}[\boldsymbol{\mu}(t)]}{\delta \boldsymbol{\mu}}\bigg)\,\textnormal{d}t \le \mathcal{E}[\boldsymbol{\mu}_{0}],
\end{equation}
where $A^{\ast}$ is the Legendre transform of $A$ in the second variable, given by
\begin{equation*}
    A^{\ast}\big(\boldsymbol{\mu}, \partial_{t}\boldsymbol{\mu}\big) = \int_{\Omega}\, \nabla \boldsymbol{\varphi}^{\ast} : M(\boldsymbol{\mu}) \nabla \boldsymbol{\varphi}^{\ast} \,\mathrm{d}x,
\end{equation*}
where $\boldsymbol{\varphi}^{\ast}$ is the solution to the following Poisson problem:
\begin{equation}\label{1-10}
    \partial_{t}\boldsymbol{\mu} - \nabla \cdot \big[M(\boldsymbol{\mu}) \nabla \boldsymbol{\varphi}^{\ast}\big] = 0 \text{ in } \Omega,\quad \nabla \boldsymbol{\varphi}^{\ast} \cdot \textnormal{\textbf{n}} = 0 \text{ on } \partial \Omega.
\end{equation}    
\end{lemma}

\begin{proof}[Proof of Lemma~\ref{EDI-lemma}]
The proof relies on the following Young's inequality coming from the definition of Legendre transform:
\begin{equation*}
    -\frac{\mathrm{d}}{\mathrm{d}t}\mathcal{E}[\boldsymbol{\mu}]\ud t = \bigg<\partial_{t}\boldsymbol{\mu}, -\frac{\delta \mathcal{E}[\boldsymbol{\mu}]}{\delta \boldsymbol{\mu}}\bigg> \le \frac{1}{2}A^{\ast}\big(\boldsymbol{\mu}, \partial_{t}\boldsymbol{\mu}\big) + \frac{1}{2}A\bigg(\boldsymbol{\mu}, -\frac{\delta \mathcal{E}[\boldsymbol{\mu}]}{\delta \boldsymbol{\mu}}\bigg),
\end{equation*}
where the equality holds if and only if $\boldsymbol{\mu}$ satisfy~\eqref{1-1}. Therefore, one can obtain an equivalent characterization of~\eqref{1-1} by integrating the above inequality over $(0, T)$, which gives us~\eqref{EDI}.
\end{proof}

\subsection{Transport metric}
In this part, we construct the transport metric~\eqref{SSD} based on the form of energy dissipation in~\eqref{EDI}. In the context of the metric space with a distance $D_{\ast}(\boldsymbol{\mu}, \boldsymbol{\nu})$, one has the following energy inequality characterization of curves of maximal slope, i.e., gradient flows (see \cite[Remark 1.3.3]{Ambrosio2005}):
\begin{equation}\label{EDI-M}
    \mathcal{E}[\boldsymbol{\mu}(T)] + \int_{0}^{T} \frac{1}{2} |\dot{\boldsymbol{\mu}}(t)|^{2} + \frac{1}{2}\big|\partial\mathcal{E}[\boldsymbol{\mu}(t)]\big|^{2}\,\textnormal{d}t \le \mathcal{E}[\boldsymbol{\mu}_{0}],
\end{equation}
where $|\dot{\boldsymbol{\mu}}(t)|$ and $\big|\partial\mathcal{E}[\boldsymbol{\mu}(t)]\big|$ are metric derivative and metric slope, respectively, defined by
\begin{equation*}
    |\dot{\boldsymbol{\mu}}(t)| = \lim_{\varepsilon \to 0} \frac{D_{\ast}(\boldsymbol{\mu}(t + \varepsilon), \boldsymbol{\mu}(t))}{\varepsilon},\quad \big|\partial\mathcal{E}[\boldsymbol{\mu}]\big| = \limsup_{\boldsymbol{\nu} \to \boldsymbol{\mu}} \frac{\big(\mathcal{E}[\boldsymbol{\mu}] - \mathcal{E}[\boldsymbol{\nu}]\big)_{+}}{D_{\ast}(\boldsymbol{\mu}, \boldsymbol{\nu})},
\end{equation*}
where $(\alpha)_{+} = \max\{0, \alpha\}$. For discussion on these concepts, please refers to \cite[Chapter 1]{Ambrosio2005}.

Now, by comparing~\eqref{EDI-M} to the EDI form~\eqref{EDI}, we can see that $A$ and $A^{\ast}$ in~\eqref{EDI} correspond to the form of metric slope and metric derivative, respectively. Hence, one can expect that
\begin{equation*}
    A^{\ast}\big(\boldsymbol{\mu}(t), \partial_{t}\boldsymbol{\mu}(t)\big)^{\frac{1}{2}} = \lim_{\varepsilon \to 0} \frac{D_{\ast}(\boldsymbol{\mu}(t + \varepsilon), \boldsymbol{\mu}(t))}{\varepsilon}.
\end{equation*}
By integrating the metric derivative $A^{\ast}$ over curves $\boldsymbol{\omega}$, we can formulate the transport metric by using the following least action principal:
\begin{equation*}
    D_{\ast}(\boldsymbol{\mu}, \boldsymbol{\nu}) = \min_{\boldsymbol{\omega}} \bigg\{\int_{0}^{1}A^{\ast}\big(\boldsymbol{\omega}(t), \partial_{t}\boldsymbol{\omega}(t)\big)\,\mathrm{d}t,\quad \boldsymbol{\omega}(0) = \boldsymbol{\mu},~\boldsymbol{\omega}(1) = \boldsymbol{\nu}\bigg\}^{\frac{1}{2}}.
\end{equation*}
Let us prove that $D_{\ast}(\boldsymbol{\mu}, \boldsymbol{\nu})$ is exactly the transport metric in~\eqref{SSD}.

\begin{proposition}\label{equi-dist} Given $\boldsymbol{\mu}, \boldsymbol{\nu}$, we have $D(\boldsymbol{\mu}, \boldsymbol{\nu}) = D_{\ast}(\boldsymbol{\mu}, \boldsymbol{\nu})$.
\end{proposition}

\begin{proof}
The following proof is based on Otto and Westdickenberg \cite{Otto2005}. By writing down the form of metric derivative, we have
\begin{equation}\label{acc}
    \begin{aligned}
        \int_{0}^{1} A^{\ast}\big(\boldsymbol{\omega}(t), \partial_{t}\boldsymbol{\omega}(t)\big)\,\mathrm{d}t &= \int_{0}^{1}\int_{\Omega} \,\nabla \boldsymbol{\varphi} : M(\boldsymbol{\omega}) \nabla \boldsymbol{\varphi}\,\mathrm{d}x\mathrm{d}t\\
        &\quad \text{ where } \partial_{t}\boldsymbol{\omega} - \nabla \cdot [M(\boldsymbol{\omega}) \nabla \boldsymbol{\varphi}] = 0 \text{ in } \Omega \times (0, 1).
    \end{aligned}
\end{equation}

\noindent
i) Because the minimal action $D_{\ast}(\boldsymbol{\mu}, \boldsymbol{\nu})$ is the special case of the Benamou-Brenier dynamical formulation $D(\boldsymbol{\mu}, \boldsymbol{\nu})$ when $\textbf{u} = -\nabla \boldsymbol{\varphi}$, we can show that the minimal action provide a upper bound for the Benamou-Brenier dynamical formulation, i.e., $D(\boldsymbol{\mu}, \boldsymbol{\nu}) \le D_{\ast}(\boldsymbol{\mu}, \boldsymbol{\nu})$.
\vspace{6pt}

\noindent
ii) The converse can be shown as follows. For any given $\boldsymbol{\omega}$ and any admissible pair $(\boldsymbol{\omega}, \textbf{u})$ that satisfies the continuity equation constraint, we can always find a potential $\boldsymbol{\varphi}$ that solves the following elliptic equation for any time $t \in (0, 1)$:
\begin{equation*}
     -\nabla \cdot \big[M(\boldsymbol{\omega}(t))\nabla \boldsymbol{\varphi}(t)\big] = \nabla \cdot \big[M(\boldsymbol{\omega}(t))\textbf{u}(t)\big] \text{ in } \Omega.
\end{equation*}
By using an integration by parts and the Neumann boundary condition~\eqref{1-10} on $\boldsymbol{\varphi}$, we have
\begin{equation*}
    \int_{0}^{1}\int_{\Omega} \,\big[M(\boldsymbol{\omega})\textbf{u} + M(\boldsymbol{\omega})\nabla \boldsymbol{\varphi}\big] : \nabla \boldsymbol{\varphi}\,\mathrm{d}x\mathrm{d}t = 0.
\end{equation*}
By the Cauchy–Schwarz inequality (see, for example, \cite[Lemma 3.3.1]{Boffi2013}), we have
\begin{equation*}
    \begin{aligned}
        \int_{\Omega}\,\nabla \boldsymbol{\varphi} : M(\boldsymbol{\omega})\nabla \boldsymbol{\varphi} &\,\mathrm{d}x = - \int_{\Omega} \,\nabla \boldsymbol{\varphi} : M(\boldsymbol{\omega})\textbf{u}\,\mathrm{d}x\\
        &\le \Bigg(\int_{\Omega}\, \textbf{u} : M(\boldsymbol{\omega})\textbf{u} \,\mathrm{d}x\Bigg)^{\frac{1}{2}} \cdot \Bigg(\int_{\Omega}\, \nabla \boldsymbol{\varphi} : M(\boldsymbol{\omega})\nabla \boldsymbol{\varphi}\,\mathrm{d}x\Bigg)^{\frac{1}{2}},
    \end{aligned}
\end{equation*}
which implies that
\begin{equation*}
    \int_{0}^{1}\int_{\Omega}\, \nabla \boldsymbol{\varphi} : M(\boldsymbol{\omega})\nabla \boldsymbol{\varphi}\,\mathrm{d}x\mathrm{d}t \le \int_{0}^{1}\int_{\Omega} \,\textbf{u} : M(\boldsymbol{\omega})\textbf{u} \,\mathrm{d}x\mathrm{d}t,
\end{equation*}
and therefore $D_{\ast}(\boldsymbol{\mu}, \boldsymbol{\nu}) \le D(\boldsymbol{\mu}, \boldsymbol{\nu})$.
\end{proof}

\subsection{Minimizing movements}
By considering the EDI form~\eqref{EDI} and the construction of distances (Theorem~\ref{equi-dist}), we can interpret~\eqref{1-1} as the gradient flow with respect to the distance $D(\boldsymbol{\mu}, \boldsymbol{\nu})$. Hence, the minimizing movement scheme~\eqref{JKOMM}:
\begin{equation*}
    \boldsymbol{\mu}^{k + 1} = \arg\min_{\boldsymbol{\mu}} \tau \mathcal{E}[\boldsymbol{\mu}] + \frac{1}{2}D(\boldsymbol{\mu}, \boldsymbol{\mu}^{k})^{2}
\end{equation*}
is a consistent variational formulation to compute~\eqref{1-1}.

In the rest of this part, we introduce a change-of-variable technique in the transport distance \cite{Benamou2000, PDFB2024, Savare2016} to obtain an equivalent convex optimization problem. Consider the so-called “momentum":
\begin{equation*}
    \textbf{m} = M(\boldsymbol{\omega})\textbf{u}.
\end{equation*}
By taking this expression of $\textbf{m}$ into~\eqref{SSD}, we have the following reformulation of the distance.
\begin{lemma}[\cite{Savare2016}] Let $D(\boldsymbol{\mu}, \boldsymbol{\nu})$ be the distance defined in~\eqref{SSD}, we have
\begin{equation}\label{SSDdd}
    \begin{aligned}
        \frac{1}{2}D(\boldsymbol{\mu}, \boldsymbol{\nu})^{2} = \min_{\boldsymbol{\omega}, \textnormal{\textbf{m}}} \Bigg\{\int_{0}^{1}\int_{\Omega}\, f\big(M(\boldsymbol{\omega}), \textnormal{\textbf{m}}\big)\,\mathrm{d}x \mathrm{d}t &: \boldsymbol{\omega}(0) = \boldsymbol{\mu} \text{ and } \boldsymbol{\omega}(1) = \boldsymbol{\nu},\\
        &\partial_{t}\boldsymbol{\omega} + \nabla\cdot \textnormal{\textbf{m}} = 0 \text{ in } \Omega \times (0, 1)\Bigg\}.
    \end{aligned}
\end{equation}
where $\textnormal{\textbf{m}} = [m_{\alpha, \beta}]_{\alpha, \beta = 1}^{n, d}$ is minimized among matrix-valued fields with vanishing normal fluxes on the boundary such that $[m_{\alpha, \beta}]_{\beta = 1}^{d} \cdot \textnormal{\textbf{n}} = 0$ on $\partial \Omega$ for all $\alpha$, and $f$ is an action function defined by
\begin{equation*}
    f(M, \textnormal{\textbf{m}}) = 
    \begin{cases}
        \frac{1}{2}\textnormal{\textbf{m}} : M^{\dagger}\textnormal{\textbf{m}} &\text{if $\textnormal{\textbf{m}} = M\textnormal{\textbf{u}}$ for some $\textnormal{\textbf{u}}$}\\
        &\text{and $0 \preceq_{\mathcal{S}_{+}} M$,}\\
        \infty &\text{otherwise.}
    \end{cases}
\end{equation*}
and $M^{\dagger}$ denotes the pseudo inverse of $M$.
\end{lemma}

Finally, by incorporating~\eqref{SSDdd} into the minimizing movement~\eqref{JKOMM}, it suffices to solve the following convex optimization problem
\begin{equation}\label{JKO}
    \begin{aligned}
        &\min_{\boldsymbol{\mu}, \boldsymbol{\omega}, \textbf{m}} \tau \mathcal{E}[\boldsymbol{\mu}] + \int_{0}^{1}\int_{\Omega} \,f\big(M(\boldsymbol{\omega}), \textbf{m}\big)\,\mathrm{d}x \mathrm{d}t\\
        &\text{ s.t. } \partial_{t}\boldsymbol{\omega} + \nabla \cdot \textbf{m} = 0 \text{ in } \Omega \times (0, 1),\quad\boldsymbol{\omega}(0) = \boldsymbol{\mu}^{k},~\boldsymbol{\omega}(1) = \boldsymbol{\mu} \text{ in } \Omega,
    \end{aligned}
\end{equation}
where $\textbf{m} = [m_{\alpha, \beta}]_{\alpha, \beta = 1}^{n, d}$ subjects to the vanishing normal fluxes on the boundary.

\section{Conforming approximation on staggered grids}\label{sec:2}
In this section, we construct the staggered grid approximation of scheme~\eqref{JKO}, and then discuss its structure-preserving properties. The discrete saddle point form, which is the final form of the numerical scheme, is summarized in the Section~\ref{sec:summary}.

\subsection{Single-step approximation}
We first introduce a single-step approximation to reduce the computation cost which has been considered in \cite{Li2019, Carrillo2023, PDFB2024}. Assume that $\textbf{m}$ is constant over inner-time $(0, 1)$, we approximate $\boldsymbol{\mu}^{k + 1}$ using the optimal $\boldsymbol{\mu}$ that minimizes,
\begin{equation}\label{single}
    \begin{aligned}
        &\min_{\boldsymbol{\mu}, \textbf{m}} \tau \mathcal{E}[\boldsymbol{\mu}] + \int_{\Omega}\, f\Big(M\big(\tfrac{1}{2}(\boldsymbol{\mu} + \boldsymbol{\mu}^{k})\big), \textbf{m}\Big)\,\mathrm{d}x\\
        &\text{ s.t. } \boldsymbol{\mu} - \boldsymbol{\mu}^{k} + \nabla \cdot \textbf{m} = 0 \text{ in } \Omega,\quad \boldsymbol{b}_{0} \le \mathcal{M}\boldsymbol{\mu} \le \boldsymbol{b}_{1} \text{ in } \Omega,
    \end{aligned}
\end{equation}
where $\mathcal{M}$ is a matrix of the size $\mathbb{R}^{p \times n}$ for given $p$, and $\boldsymbol{b}_{0} = [b_{0, \beta}]_{\beta = 1}^{p}, \boldsymbol{b}_{1} = [b_{1, \beta}]_{\beta = 1}^{p}$ are global bounds, which are set by specific form of the model to keep $M(\boldsymbol{\mu})$ positive definite.

\begin{remark}
This box constraint is necessary to ensure bounded preservation at the discrete level (see \cite[Equation (11)]{PDFB2024}).
\end{remark}

\begin{remark}
It has been shown that single-step approximations in the dynamical transport distance $D(\boldsymbol{\mu}, \boldsymbol{\nu})$ does not influence the resolution order of JKO schemes \cite{Li2019, Carrillo2023, Li2023}.
\end{remark}

\begin{remark}\label{box5}
Compared to the scalar case in \cite{PDFB2024}, the bound of the solution may be coupled for~\eqref{1-1}. For example, consider the following mobility matrix from \cite{Carrillo2022} (where $\sigma$ is a positive constant which is referred to as the saturation level):
\begin{equation*}
    M(\boldsymbol{\mu}) = \left[\begin{array}{cc}
        \mu_{1}(\sigma - (\mu_{1} + \mu_{2})) & 0 \\
        0 & \mu_{2}(\sigma - (\mu_{1} + \mu_{2}))
    \end{array}\right]
\end{equation*}
where the box constraint is given by $\mu_{0} \ge 0,~ \mu_{1} \ge 0, ~ \mu_{1} + \mu_{2} \le \sigma$. Hence, in the numerical scheme~\eqref{single} the corresponding matrix $\mathcal{M}$ and bounds are given by
\begin{equation*}
    \mathcal{M} =
    \left[\begin{array}{cc}
        1 & 0\\
        0 & 1\\
        1 & 1\\
    \end{array}\right],\quad \boldsymbol{b}_{0} =
    \left[\begin{array}{c}
        0\\
        0\\
        -\infty
    \end{array}\right],\quad \boldsymbol{b}_{1} =
    \left[\begin{array}{c}
        \infty\\
        \infty\\
        \sigma
    \end{array}\right].
\end{equation*}
\end{remark}

\subsection{Saddle point form}
The main difficulty in solving the variational problem~\eqref{single} is the non-smoothness of the action function $f$. To remedy this, we consider proximal gradient methods for non-smooth optimization problems (see, for example, \cite{Chambolle2011, Pock2011, Chambolle2016, Malitsky2018}). In this subsection, we derive a saddle point form of~\eqref{single} which is critical in non-smooth optimization methods.

Closely following the discussion on the scalar action functions \cite{Benamou2000, Papadakis2013, PDFB2024}, we show the following theorem, which states that the action function $f$ is a support function of a convex set, or the Legendre transform of $f$ is the indicator function of a convex set.
\begin{theorem}\label{dual_lemma} Let $M \in \mathbb{R}^{n \times n}$ be a symmetric matrix, and $\textnormal{\textbf{m}} \in \mathbb{R}^{n \times d}$, one has
\begin{equation*}
    \begin{aligned}
        f(M, \textnormal{\textbf{m}}) = \max_{(Q, \boldsymbol{q}) \in K} M : Q + \textnormal{\textbf{m}} : \boldsymbol{q},
    \end{aligned}
\end{equation*}
where $K$ is a closed convex set defined by
\begin{equation}\label{kh}
    K = \bigg\{(Q, \boldsymbol{q}) \in \mathbb{R}^{n \times n}_{\textnormal{sym}} \times \mathbb{R}^{n \times d} : Q + \frac{\boldsymbol{q} \boldsymbol{q}^{T}}{2} \preceq_{\mathcal{S}_{+}} 0\bigg\},
\end{equation}
where $\mathbb{R}^{n \times n}_{\textnormal{sym}}$ is the space of symmetrical $n \times n$ matrices, and $\preceq_{\mathcal{S}_{+}}$ is the order relation with respect to the cone of positive semi-definite matrix $\mathcal{S}_{+}$. 
\end{theorem}
We postpone the proof of this theorem to Section~\ref{sec:supp}. By using the support function form in Theorem~\ref{dual_lemma}, we have
\begin{equation*}
    \begin{aligned}
        \int_{\Omega} \,f\Big(M\big(\tfrac{1}{2}(\boldsymbol{\mu} + \boldsymbol{\mu}^{k})\big), \textbf{m}\Big) \,\mathrm{d}x = \sup \bigg\{\int_{\Omega}\, M\big(\tfrac{1}{2}(\boldsymbol{\mu} &+ \boldsymbol{\mu}^{k})\big) : Q + \textbf{m} : \boldsymbol{q} \,\mathrm{d}x:\\
        &(Q(x), \boldsymbol{q}(x)) \in K\quad \forall x \in \Omega\bigg\}.
    \end{aligned}
\end{equation*}
Hence, by using this support function characterization, we can write the variational problem~\eqref{single} in the following saddle point form,
\begin{equation}\label{obj_phi}
    \begin{aligned}
        &\min_{\boldsymbol{\mu}, \textbf{m}}\max_{Q, \boldsymbol{q}} \tau \mathcal{E}[\boldsymbol{\mu}] + \int_{\Omega}\, M\big(\tfrac{1}{2}\big(\boldsymbol{\mu} + \boldsymbol{\mu}^{k}\big)\big) : Q + \textbf{m} : \boldsymbol{q}\,\mathrm{d}x,\\
        &\text{ s.t. } \boldsymbol{\mu} - \boldsymbol{\mu}^{k} + \nabla \cdot \textbf{m} = 0,\quad \boldsymbol{b}_{0} \le \mathcal{M}\boldsymbol{\mu} \le \boldsymbol{b}_{1},\quad Q + \frac{\boldsymbol{q} \boldsymbol{q}^{T}}{2} \preceq_{\mathcal{S}_{+}} 0.
    \end{aligned}
\end{equation}
where $\textbf{m} = [m_{\alpha, \beta}]_{\alpha, \beta = 1}^{n, d}$ subjects to the vanishing normal fluxes on the boundary.

\subsection{Staggered grid approximation}
In this subsection, we construct a staggered grid approximation of~\eqref{single} and proceed to derive the discrete min-max saddle point form corresponding to~\eqref{obj_phi}. This construction is a multi-species version of the staggered grid approximation in our previous work \cite{PDFB2024}.

We cover the domain $\Omega$ with non-overlapping uniform boxes $\Omega_{\textbf{i}}$ defined by
\begin{equation*}
    \Omega_{\textbf{i}} = \Big\{x : \big|x - x_{\textbf{i}}\big|_{\infty} \le \tfrac{1}{2}h \Big\},
\end{equation*}
where $h$ is the grid spacing and $x_{\textbf{i}}$ is the center of the box $\Omega_{\textbf{i}}$. To write the scheme in a finite-difference stencil, we adopt a global indices to location of the cube $\Omega_{\textbf{i}}$:
\begin{equation*}
    \mathfrak{T} = \Big\{\textbf{i} = \big(i_{1},..., i_{d}\big)\Big\}.
\end{equation*}
Also, the boundary of each box $\Omega_{\textbf{i}}$ can be partitioned into $2d$-rectangles $\bigcup_{q = 1}^{d}\Gamma_{\textbf{i} \pm \frac{1}{2}e_{q}}^{q}$ of the dimension $d - 1$,
where $\Gamma_{\textbf{i} \pm \frac{1}{2}e_{q}}^{q}$ corresponds to two bounding interfaces orthogonal to the standard basis vector $e_{q}$ pointing in the $q$-th direction. We use the following indices to locate $\Gamma_{\textbf{i} \pm \frac{1}{2}e_{q}}^{q}$:
\begin{equation*}
    \mathfrak{J} = (\mathfrak{J}^{q})_{q = 1}^{d} \text{ where } \mathfrak{J}^{q} = \Big\{\textbf{i} = \big(i_{1},..., i_{q} \pm \tfrac{1}{2},... i_{d}\big)\Big\}.
\end{equation*}

Now, we let $\mathfrak{T}(\mathbb{R})$ and $\mathfrak{J}(\mathbb{R})$ be the discrete scalar field and flux field such that
\begin{equation*}
    \mathfrak{T}(\mathbb{R}) = \Big\{\sigma : \mathfrak{T} \to \mathbb{R}\Big\},\quad \mathfrak{J}(\mathbb{R}) = \Big\{\boldsymbol{\sigma} = (\sigma^{1}, ..., \sigma^{d}) : \quad \sigma^{q} : \mathfrak{J}^{q} \to \mathbb{R},\quad \forall 1 \le q \le d\Big\}.
\end{equation*}
where we use $\sigma_{\textbf{i}}$ or $\sigma_{\textbf{j}}^{q}$ to denote the value of the filed on the $\textbf{i}$-th box or $\textbf{j}$-th interfaces, respectively, for scalar filed $\sigma \in \mathfrak{T}(\mathbb{R})$ or vector field $\boldsymbol{\sigma} = (\sigma^{1}, ..., \sigma^{d}) \in \mathfrak{J}(\mathbb{R})$. We introduce the space of discrete momentum with vanishing normal fluxes of $\textbf{m}$ on the boundary:
\begin{equation*}
    \mathfrak{J}_{0}(\mathbb{R}) = \Big\{\boldsymbol{\sigma} = (\sigma^{1}, ..., \sigma^{d}) \in \mathfrak{J}(\mathbb{R}) :\quad \sigma_{\textbf{j}}^{q} = 0 \text{ when } \Gamma_{\textbf{j}}^{q} \subseteq \partial \Omega\quad \forall 1 \le q \le d\Big\},
\end{equation*}
which encodes the non-flux boundary condition into $\textbf{m}$ in the optimization.

\subsubsection{Construction}
In this part, we construct the approximation of the variational problem~\eqref{single} on the staggered grid $\mathfrak{T}(\mathbb{R})$ and $\mathfrak{J}(\mathbb{R})$ constructed above.

\vspace{6pt}
\noindent
\textit{1. Constraint}: We consider the following discrete admissible set in \eqref{single}:
\begin{equation}\label{addsmi}
    \begin{aligned}
        H = \Big\{(\boldsymbol{\mu}, \textbf{m}) \text{ such that } &(\mu_{\alpha}, \textbf{m}_{\alpha}) \in \mathfrak{T}(\mathbb{R}) \times \mathfrak{J}_{0}(\mathbb{R}) : \\
        &\mu_{\alpha} - \mu_{\alpha}^{k} + \mathcal{A}\textbf{m}_{\alpha} = 0, \quad \forall 1 \le \alpha \le n,\\
        &b_{0, \beta} \le \sum_{\alpha = 1}^{n} \mathcal{M}_{\beta \alpha}\mu_{\alpha, \textbf{i}} \le b_{1, \beta}, \quad \forall \textbf{i} \in \mathfrak{T} \text{ and } 1 \le \beta \le p\Big\}.
    \end{aligned}
\end{equation}
where $\mathcal{A}$ is the discrete divergence operator
\begin{equation*}
    \mathcal{A} : \mathfrak{J}(\mathbb{R}) \to \mathfrak{T}(\mathbb{R}),\quad (\mathcal{A}\boldsymbol{\sigma})_{\textbf{i}} = \frac{1}{h}\sum_{q = 1}^{d} \Big(\sigma_{\textbf{i} + \frac{1}{2}e_{q}}^{q} - \sigma_{\textbf{i} - \frac{1}{2}e_{q}}^{q}\Big) \text{ for } \boldsymbol{\sigma} = (\sigma^{1}, ..., \sigma^{d}).
\end{equation*}

\noindent
\textit{2. Action functional}: We consider a centered quadrature rule for the action functional
\begin{equation*}
    \int_{\Omega} \,f\Big(M\big(\tfrac{1}{2}(\boldsymbol{\mu} + \boldsymbol{\mu}^{k})\big), \textbf{m}\Big)\mathrm{d}x \approx \sum_{\textbf{i} \in \mathfrak{T}}f\Big(M\big(\tfrac{1}{2}(\boldsymbol{\mu}_{\textbf{i}} + \boldsymbol{\mu}^{k}_{\textbf{i}})\big), (\mathcal{I}\textbf{m})_{\textbf{i}}\Big)h^{d}
\end{equation*}
where $\mathcal{I}$ is the interpolation operator
\begin{equation*}
    \mathcal{I} : \mathfrak{J}(\mathbb{R}) \to \mathfrak{T}(\mathbb{R})^{d},\quad (\mathcal{I}\boldsymbol{\sigma})_{\textbf{i}} = \left(\frac{\sigma_{\textbf{i} + \frac{1}{2}e_{1}}^{1} + \sigma_{\textbf{i} - \frac{1}{2}e_{1}}^{1}}{2}, \cdots, \frac{\sigma_{\textbf{i} + \frac{1}{2}e_{d}}^{d} + \sigma_{\textbf{i} - \frac{1}{2}e_{d}}^{d}}{2}\right).
\end{equation*}
Here, by a sight abuse of notation, interpolation $\mathcal{I}\textbf{m}$ is a row-wise (specie-wise) interpolation when operated on a discrete matrix field in the space $\mathfrak{J}(\mathbb{R})^{n}$.

\vspace{6pt}
\noindent
\textit{3. Free energy functional}: For the free energy functional~\eqref{energy}, we consider a centered quadrature rule for the internal energy and the confinement potential, and a trapezoidal rule for the Dirichlet energy, which has been discussed in \cite{Carrillo2023}:
\begin{equation*}
    \mathcal{E}[\boldsymbol{\mu}] \approx \mathcal{E}_{h}[\boldsymbol{\mu}] = \sum_{\textbf{i}\in\mathfrak{T}} \Big(U(\boldsymbol{\mu}_{\textbf{i}}) + V(x_{\textbf{i}}) \cdot \boldsymbol{\mu}_{\textbf{i}} \Big) h^{d} + \sum_{q = 1}^{d}\sum_{\textbf{j} \in \mathfrak{J}^{q}} \frac{\varepsilon^{2}}{2} |\partial_{q, h} \boldsymbol{\mu}|^{2}_{\textbf{j}} h^{d},
\end{equation*}
where $\partial_{q, h}$ is the discrete partial derivative in the $q$-th direction defined by
\begin{equation*}
    \big(\partial_{q, h} \boldsymbol{\mu}\big)_{\textbf{j}} = \frac{\boldsymbol{\mu}_{\textbf{j} + \frac{1}{2}e_{q}} - \boldsymbol{\mu}_{\textbf{j} - \frac{1}{2}e_{q}}}{h}.
\end{equation*}
For the discrete Dirichlet energy $\sum_{q = 1}^{d}\sum_{\textbf{j} \in \mathfrak{J}^{q}} \frac{\varepsilon^{2}}{2} |\partial_{q, h} \boldsymbol{\mu}|^{2}_{\textbf{j}}$, its gradient corresponds to a centered approximation of the Laplacian $-\varepsilon^{2}\Delta_{h}\boldsymbol{\mu}$ with a Neumann boundary condition (see \cite{Carrillo2023}).

\subsection{Summary}\label{sec:summary}
In summary, we solve the following optimization in each time step
\begin{equation}\label{dsingle}
    \min_{\boldsymbol{\mu}, \textbf{m}} \tau \mathcal{E}_{h}[\boldsymbol{\mu}] + \sum_{\textbf{i} \in \mathfrak{T}}f\Big(M\big(\tfrac{1}{2}(\boldsymbol{\mu}_{\textbf{i}} + \boldsymbol{\mu}^{k}_{\textbf{i}})\big), (\mathcal{I}\textbf{m})_{\textbf{i}}\Big)h^{d} + \iota_{H}(\boldsymbol{\mu}, \textbf{m}),
\end{equation}
where $\iota_{A}(\cdot)$ is the indicator function of a set $A$ such that
\begin{equation*}
    \iota_{A}(\textbf{x}) = 
    \begin{cases}
        0 & \text{ if } \textbf{x} \in A\\
        \infty & \text{ otherwise}.
    \end{cases}
\end{equation*}
One can show that optimizing the discrete objective functional in~\eqref{dsingle} preserves the structure of~\eqref{1-1} at a fully discrete level.
\begin{theorem}[Structure-preserving property]\label{structure}Let $\boldsymbol{\mu}^{k + 1} = [\mu_{\alpha}^{k + 1}]_{\alpha = 1}^{n}$ be the solution to~\eqref{dsingle}, one has
\begin{enumerate}
    \item Energy dissipation: $\mathcal{E}_{h}[\boldsymbol{\mu}^{k + 1}] \le \mathcal{E}_{h}[\boldsymbol{\mu}^{k}]$.

    \item Mass conservation: $\sum_{\textnormal{\textbf{i}} \in \mathfrak{T}} \mu_{\alpha, \textnormal{\textbf{i}}}^{k + 1} = \sum_{\textnormal{\textbf{i}} \in \mathfrak{T}} \mu_{\alpha, \textnormal{\textbf{i}}}^{k}\quad \forall \alpha$.

    \item Bound preservation: $0 \preceq_{\mathcal{S}_{+}} M(\boldsymbol{\mu}^{k + 1}_{\textnormal{\textbf{i}}}) \quad \forall \textnormal{\textbf{i}} \in \mathfrak{T}$.
\end{enumerate}
\end{theorem}
The proof is identical to that of \cite[Theorem 1]{PDFB2024} with minor modifications.

By applying the support function form (Theorem~\ref{dual_lemma}) to the action function on each grid, we have the following discrete saddle point form corresponding to~\eqref{obj_phi}, which is our numerical scheme.
\begin{scheme}\rm
Let $\tau$ be the time step size and $h$ be the spatial grid size. Given $\boldsymbol{\mu}^{k}$, we approximate the solution $\boldsymbol{\mu}^{k + 1}$ of~\eqref{JKOMM}, and therefore the solution $\boldsymbol{\mu}(t)$ of~\eqref{1-1} at $t = (k + 1)\tau$, by finding the optimal $\boldsymbol{\mu}$ that solves
\begin{equation}\label{obj_sad}
    \begin{aligned}
        \min_{\boldsymbol{\mu}, \textbf{m}} \max_{Q, \boldsymbol{q}} \tau \widehat{\mathcal{E}}_{h}[\boldsymbol{\mu}] + \sum_{\textbf{i} \in \mathfrak{T}} M\big(\tfrac{1}{2}(\boldsymbol{\mu}_{\textbf{i}} + \boldsymbol{\mu}^{k}_{\textbf{i}})\big) : Q_{\textbf{i}} &+ (\mathcal{I}\textbf{m})_{\textbf{i}} : \boldsymbol{q}_{\textbf{i}}\\
        &- \sum_{\textbf{i} \in \mathfrak{T}}\iota_{K}\big(Q_{\textbf{i}}, \boldsymbol{q}_{\textbf{i}}\big) + \iota_{H}(\boldsymbol{\mu}, \textbf{m}),
    \end{aligned}
\end{equation}
where $\widehat{\mathcal{E}}_{h} = \mathcal{E}_{h}/h^{d}$ denotes the normalized energy functional.
\end{scheme}

To improve the practical application of the numerical scheme~\eqref{obj_sad}, we introduce a nonlinear preconditioning proposed in \cite{PDFB2024}. Let us consider a convex splitting of the energy, and denote,
\begin{equation*}
    \widehat{\mathcal{E}}_{h}[\boldsymbol{\mu}] = \mathcal{U}[\mathcal{K}\boldsymbol{\mu}] + \big(\widehat{\mathcal{E}}_{h}[\boldsymbol{\mu}] - \mathcal{U}[\mathcal{K}\boldsymbol{\mu}]\big) = \mathcal{U}[\mathcal{K}\boldsymbol{\mu}] + \mathcal{V}[\boldsymbol{\mu}],
\end{equation*}
where $\mathcal{U}$ is the preconditioner which requires to be convex, and $\mathcal{K}$ is a linear operator. By applying the Legendre transform to $\mathcal{U}$, we have the following preconditioned numerical scheme.
\begin{scheme}\rm
Let $\tau$ be the time step size and $h$ be the spatial grid size. Given $\boldsymbol{\mu}^{k}$, we approximate the solution $\boldsymbol{\mu}^{k + 1}$ of~\eqref{JKOMM}, and therefore the solution $\boldsymbol{\mu}(t)$ of~\eqref{1-1} at $t = (k + 1)\tau$, by finding the optimal $\boldsymbol{\mu}$ that solves
\begin{equation}\label{obj_split}
    \begin{aligned}
        \min_{\boldsymbol{\mu}, \textbf{m}} \max_{Q, \boldsymbol{q}, \boldsymbol{\nu}} \tau \mathcal{V}[\boldsymbol{\mu}] &+ \sum_{\textbf{i} \in \mathfrak{T}} M\big(\tfrac{1}{2}(\boldsymbol{\mu}_{\textbf{i}} + \boldsymbol{\mu}^{k}_{\textbf{i}})\big) : Q_{\textbf{i}} + (\mathcal{I}\textbf{m})_{\textbf{i}} : \boldsymbol{q}_{\textbf{i}}\\
        &+ (\mathcal{K}\boldsymbol{\mu})_{\textbf{i}} \cdot \boldsymbol{\nu}_{\textbf{i}} - \tau\mathcal{U}^{\ast}[\boldsymbol{\nu}/\tau]  - \sum_{\textbf{i} \in \mathfrak{T}}\iota_{K}(Q_{\textbf{i}}, \boldsymbol{q}_{\textbf{i}}) + \iota_{H}(\boldsymbol{\mu}, \textbf{m}),
    \end{aligned}
\end{equation}
where $\mathcal{U}^{\ast}$ is the convex conjugate of $\mathcal{U}$.
\end{scheme}

\section{PDFB splitting optimization}\label{sec:3}
In this section, we present our strategy to solve the saddle point problem~\eqref{obj_split}. Here we denote the primal variable and the dual variable as $\textbf{x}$ and $\textbf{y}$ respectively. With this notation, the optimization problem~\eqref{obj_split} can be then compactly written as follows,
\begin{equation}\label{oframe}
    \min_{\textbf{x} = (\boldsymbol{\mu}, \textbf{m})} \max_{\textbf{y} = (Q, \boldsymbol{q}, \boldsymbol{\nu})}\Phi(\textbf{x}, \textbf{y}) - F(\textbf{y}) + G(\textbf{x}).
\end{equation}
where
\begin{equation}\label{phi_form}
    \Phi(\textbf{x}, \textbf{y}) = \tau \mathcal{V}[\boldsymbol{\mu}] + \sum_{\textbf{i} \in \mathfrak{T}} M\big(\tfrac{1}{2}(\boldsymbol{\mu}_{\textbf{i}} + \boldsymbol{\mu}^{k}_{\textbf{i}})\big) : Q_{\textbf{i}} + (\mathcal{I}\textbf{m})_{\textbf{i}} : \boldsymbol{q}_{\textbf{i}} + (\mathcal{K}\boldsymbol{\mu})_{\textbf{i}} \cdot \boldsymbol{\nu}_{\textbf{i}}
\end{equation}
\begin{equation*}
    G(\textbf{x}) = \iota_{H}(\boldsymbol{\mu}, \textbf{m}) \text{ and } F(\textbf{y}) = \tau\mathcal{U}^{\ast}[\boldsymbol{\nu}/\tau] + \sum_{\textbf{i} \in \mathfrak{T}}\iota_{K}(Q_{\textbf{i}}, \boldsymbol{q}_{\textbf{i}}).
\end{equation*}
The optimization problem~\eqref{oframe} belongs to the class of general convex-concave saddle point problems, where $\Phi$ is continuously differentiable while $F$ and $G$ are convex but non-smooth.

\subsection{Algorithm}
In recent years, there are several optimization algorithms (see, for example, \cite{Malitsky2018, Hamedani2018, Zhu2022}) proposed to cope with general convex-concave saddle point problems~\eqref{oframe}. One of the key ideas behind these methods is the forward-backward splitting approach, in which one combines the gradient descents/ascents of $\Phi(\textbf{x}, \textbf{y})$ with the proximal gradient updates of $G(\textbf{x})$ and $F(\textbf{y})$ in each step of the iteration. Reflection techniques in operator splitting methods can also be considered to improve the optimization algorithms \cite{Chambolle2011, Yan2016, Malitsky2018}.

In the following, we apply the PDFB splitting method developed in our previous work \cite[Section 3]{PDFB2024} to solve the saddle point problem~\eqref{oframe}.

We introduce the Jacobi matrix $\mathcal{J}(\textbf{x})$ defined by
\begin{equation*}
    \mathcal{J}(\textbf{x})\textbf{w} = \lim_{\varepsilon \to 0} \frac{\nabla_{\textbf{y}} \Phi(\textbf{x} + \varepsilon \textbf{w}, \textbf{y}) - \nabla_{\textbf{y}} \Phi(\textbf{x}, \textbf{y})}{\varepsilon},\quad \forall \textbf{w} \in \mathbb{R}^{\text{dim}(\textbf{x})}
\end{equation*}
At each iterate $(\textbf{x}^{\ell}, \textbf{y}^{\ell})$, we update ($\gamma$ and $\overline{\gamma}$ are step size parameters of the algorithm)
\begin{align*}
    \textbf{y}^{\ell + 1} &= \text{prox}_{\gamma F}\big(\textbf{y}^{\ell} + \gamma \mathcal{J}(\textbf{x}^{\ell}) (\overline{\textbf{x}}^{\ell} - \textbf{x}^{\ell}) + \gamma\nabla_{\textbf{y}} \Phi(\textbf{x}^{\ell}, \textbf{y}^{\ell})\big)\\
    \textbf{x}^{\ell + 1} &= \text{prox}_{\overline{\gamma} G}\big(\textbf{x}^{\ell} - \overline{\gamma} \nabla_{\textbf{x}}\Phi(\textbf{x}^{\ell}, \textbf{y}^{\ell + 1})\big)\\
    \overline{\textbf{x}}^{\ell + 1} &= 2\textbf{x}^{\ell + 1} - \textbf{x}^{\ell} - \overline{\gamma} \big(\nabla_{\textbf{x}}\Phi(\textbf{x}^{\ell + 1}, \textbf{y}^{\ell + 1}) - \nabla_{\textbf{x}}\Phi(\textbf{x}^{\ell}, \textbf{y}^{\ell + 1})\big)
\end{align*}
where $\text{prox}_{F}$ dentoes the proximal operator of a function defined by:
\begin{equation*}
    \text{prox}_{F}(\textbf{x}_{0}) = \arg\min_{v} F(\textbf{x}) + \frac{1}{2}\|\textbf{x} - \textbf{x}_{0}\|^{2}.
\end{equation*}
For details concerning the motivation and analysis of the algorithm, please refers to \cite{PDFB2024}.

\subsection{Implementation details}
In each step of the PDFB splitting method, one needs to the compute derivatives of $\Phi(\textbf{x}, \textbf{y})$. For~\eqref{phi_form}, the gradients are trivial to evaluate by chain rules. In the following, we mainly discuss how to compute the proximal operators $G(\textbf{x})$ and $F(\textbf{y})$. It is important to note that the dual variables $(Q, \boldsymbol{q})$ and $\boldsymbol{\nu}$ are separable when computing the proximal operator of $F(\textbf{y})$. We will discuss them as two individual subproblems, where the proximal problem associated with $\boldsymbol{\nu}$ is referred to as the proximal problem of the nonlinear preconditioner.

\subsubsection{Proximal operator of primal variable}
The proximal problem associated with $\text{prox}_{\gamma G}$ amounts to solve a box constrained quadratic optimization problem for given $(\boldsymbol{\mu}_{0}, \textbf{m}_{0}) \in \big(\mathfrak{T}(\mathbb{R}) \times \mathfrak{J}_{0}(\mathbb{R})\big)^{n}$:
\begin{equation*}
    \begin{aligned}
        &\min_{\boldsymbol{\mu}, \textbf{m}} \frac{1}{2}\|\boldsymbol{\mu} - \boldsymbol{\mu}_{0}\|^{2} + \frac{1}{2}\|\textbf{m} - \textbf{m}_{0}\|^{2}\\
        &\text{ s.t. } \mu_{\alpha} - \mu_{\alpha}^{k} + \mathcal{A}\textbf{m}_{\alpha} = 0, \quad \forall 1 \le \alpha \le n,\\
        &\qquad b_{0, \beta} \le \sum_{\alpha = 1}^{n} \mathcal{M}_{\beta \alpha}\mu_{\alpha, \textbf{i}} \le b_{1, \beta}, \quad \forall \textbf{i} \in \mathfrak{T} \text{ and } 1 \le \beta \le p.
    \end{aligned}
\end{equation*}
This problem can be solved efficiently by the primal-dual active set method \cite[page 319]{Manzoni2021}, which is also considered in our previous work \cite{PDFB2024}. Let us consider the Karush-Kuhn-Tucker (KKT) optimality conditions:
\begin{equation}\label{KKT}
    \begin{aligned}
        &\begin{cases}
        \mu_{\alpha} - \mu_{0, \alpha} + \phi_{\alpha} + \sum_{\beta = 1}^{p}\mathcal{M}_{\beta\alpha}\lambda_{\beta} = 0\\
        \textbf{m}_{\alpha} - \textbf{m}_{0, \alpha} + \mathcal{A}^{T}\phi_{\alpha} = 0\\
        \mu_{\alpha} - \mu^{k}_{\alpha} + \mathcal{A}\textbf{m}_{\alpha} = 0\\
    \end{cases}
    \forall 1 \le \alpha \le n,\\
    &\quad \text{ and } C_{\beta}(\boldsymbol{\mu}, \lambda_{\beta}) = 0,\quad \forall 1 \le \beta \le p.
    \end{aligned}
\end{equation}
where $(\phi_{\alpha})$ and $(\lambda_{\beta})$ are the Lagrangian multipliers associated with the linear constraints and box constraints, respectively, and $C_{\beta}(\boldsymbol{\mu}, \lambda_{\beta})$ is the $\beta$-th complementary function corresponding to the box constraint:
\begin{equation*}
    C_{\beta}(\boldsymbol{\mu}, \lambda_{\beta}) = \lambda_{\beta} - \bigg[\lambda_{\beta} + \varepsilon_{0}\bigg(\sum_{\alpha = 1}^{n}\mathcal{M}_{\beta \alpha}\mu_{\alpha} - b_{0, \beta}\bigg)\bigg]_{-} - \bigg[\lambda_{\beta} + \varepsilon_{0}\bigg(\sum_{\alpha = 1}^{n}\mathcal{M}_{\beta \alpha}\mu_{\alpha} - b_{1, \beta}\bigg)\bigg]_{+}
\end{equation*}
with $\varepsilon_{0} > 0$ where we fix $\varepsilon_{0} = 1$ in our computation, and $[v]_{\pm}$ stands for the positive part and negative part of the vector $v$ respectively.

We now begin to describe the primal-dual active set. In this method, we iteratively update the variables using the following three steps in each iteration.
\vspace{6pt}

\noindent
1. Given the $\ell$-th update $(\boldsymbol{\mu}^{\ell}, \textbf{m}^{\ell})$, we compute the active sets (for all $1 \le \beta \le p$):
\begin{equation*}
    A^{\beta, -}_{\ell} = \Big\{\textbf{i} : \lambda_{\beta, \textbf{i}}^{\ell} + (\mathcal{M}\boldsymbol{\mu}^{\ell}_{\textbf{i}})_{\beta} < b_{0, \beta} \Big\} \text{ and } A^{\beta, +}_{\ell} = \Big\{\textbf{i} : \lambda_{\beta, \textbf{i}}^{\ell} + (\mathcal{M}\boldsymbol{\mu}^{\ell}_{\textbf{i}})_{\beta} > b_{1, \beta}\Big\},
\end{equation*}
and inactive set $\mathfrak{T}_{\ell}^{\beta} = \mathfrak{T} - (A^{\beta, -}_{\ell}\cup A^{\beta, +}_{\ell})$ for all $\beta$ from $1$ to $p$. Here $A^{\beta, \pm}_{\ell}$ denotes the failures to satisfy the $\beta$-th box constraint at $\ell$-th iterate. In order to formulate the optimality condition~\eqref{KKT} in a sparse linear system form, we concatenate all the variables $\{(\mu_{\alpha}, \textbf{m}_{\alpha})\}$ and auxiliary variables such that
\begin{equation*}
    \textbf{u}^{\ell + 1} = \left[
    \begin{array}{c}
        \mu_{1}^{\ell + 1}\\
        \vdots\\
        \mu_{n}^{\ell + 1}\\
        \textbf{m}_{1}^{\ell + 1}\\
        \vdots\\
        \textbf{m}_{n}^{\ell + 1}\\
    \end{array}\right],\,
    \textbf{u}_{0} = \left[
    \begin{array}{c}
        \mu_{0, 1}\\
        \vdots\\
        \mu_{0, n}\\
        \textbf{m}_{0, 1}\\
        \vdots\\
        \textbf{m}_{0, n}\\
    \end{array}\right],\, \textbf{v}^{\ell + 1} = \left[
    \begin{array}{c}
        \phi^{\ell + 1}\\
        \lambda_{1, A^{1, -}_{\ell}}^{\ell + 1}\\
        \lambda_{1, A^{1, +}_{\ell}}^{\ell + 1}\\
        \vdots\\
        \lambda_{p, A^{p, -}_{\ell}}^{\ell + 1}\\
        \lambda_{p, A^{p, +}_{\ell}}^{\ell + 1}
    \end{array}\right],\, \textbf{v}_{0} = \left[
    \begin{array}{c}
        \mu^{k}_{1}\\
        \vdots\\
        \mu^{k}_{n}\\
        b_{0, 1} e_{A^{1, -}_{\ell}}\\
        b_{1, 1} e_{A^{1, +}_{\ell}}\\
        \vdots\\
        b_{0, p} e_{A^{p, -}_{\ell}}\\
        b_{1, p} e_{A^{p,+}_{\ell}}
    \end{array}\right].
\end{equation*}
where $e_{A} = (1, \cdots, 1)^{T}$ is a vector with a length equals to the cardinality number $|A|$. Herein $\phi$ is a vector of the dimension $n \times |\mathfrak{T}|$, corresponding to the Lagrangian multiplier of the discrete continuity constraints, and $\lambda_{\beta, A^{\beta, \pm}}$ are the multipliers associated with the active sets.
\vspace{6pt}

\noindent
2. We update $\textbf{u}^{\ell + 1}$ and $\textbf{v}^{\ell + 1}$ by solving the following saddle point form:
\begin{equation*}
    \left[\begin{array}{cc}
        I & \mathcal{J}^{T}_{\ell} \\
        \mathcal{J}_{\ell} & 0
    \end{array}\right] \left[\begin{array}{c}
        \textbf{u}^{\ell + 1}\\
        \textbf{v}^{\ell + 1}
    \end{array}\right] = \left[\begin{array}{c}
        \textbf{u}_{0}\\
        \textbf{v}_{0}
    \end{array}\right].
\end{equation*}
where $\mathcal{J}_{\ell}$ is given by
\begin{equation*}
    \mathcal{J}_{\ell} =
    \left[\begin{array}{cccc}
        I_{n \times n} \otimes I & I_{n \times n} \otimes \mathcal{A}\\
        \mathcal{N}_{\ell}
    \end{array}\right]
\end{equation*}
where $I$ is the identity matrix of the size $\text{dim}(\textbf{u}^{\ell + 1})$, $\otimes$ is the Kronecker product, and
\begin{equation*}
    \mathcal{N}_{\ell} = \Big[\mathcal{M}_{\beta, \alpha} P_{A_{\ell}^{\beta, -}}\Big]_{\beta, \alpha = 1}^{p, n} = 
    \left[\begin{array}{cccc}
        \mathcal{M}_{1, 1} P_{A^{1, -}_{\ell}} & & \mathcal{M}_{1, n} P_{A^{1, -}_{\ell}}\\
        \mathcal{M}_{1, 1} P_{A^{1, +}_{\ell}} & & \mathcal{M}_{1, n} P_{A^{1, +}_{\ell}}\\
         & \ddots & \\
         & \ddots & \\
        \mathcal{M}_{1, p} P_{A^{p, -}_{\ell}} &  & \mathcal{M}_{n, p} P_{A^{p, -}_{\ell}}\\
        \mathcal{M}_{1, p} P_{A^{p, -}_{\ell}} &  & \mathcal{M}_{n, p} P_{A^{p, -}_{\ell}}\\
    \end{array}\right]
\end{equation*}
where $P_{A}$ the selection matrix such that $P_{A}\lambda$ is the vector of components of $\lambda$ on the index set $A$.
\vspace{6pt}

\noindent
3. After obtaining $\textbf{u}^{\ell + 1}$ and $\textbf{v}^{\ell + 1}$ using step 2, we can conversely recover the multiplier $(\lambda_{\beta})$ by
\begin{equation}\label{eq:lam_recover}
    \lambda^{\ell + 1}_{\beta} = P^{T}_{A^{\beta, -}_{\ell}}\lambda_{\beta, A^{\beta, -}_{\ell}}^{\ell + 1} + P^{T}_{A^{\beta, +}_{\ell}}\lambda_{\beta, A^{\beta, +}_{\ell}}^{\ell + 1}.
\end{equation}
In the active set method, we repeatedly solve the above system to obtain $\textbf{u}^{\ell + 1}$ and $\textbf{v}^{\ell + 1}$, and update the active sets $A^{\beta, \pm}_{\ell + 1}$ and the matrix $\mathcal{J}_{\ell+1}$ by recovering the multiplier $\lambda_{\beta}^{\ell + 1}$ by \eqref{eq:lam_recover} until the active sets remain unchanged, i.e., $A^{\beta, \pm}_{\ell + 1} = A^{\beta, \pm}_{\ell}$ for some $\ell$.

\begin{remark} We note that, if the bound of the solution is not coupled, i.e., $\mathcal{M} = I_{n \times n}$, the proximal operator $\text{prox}_{\gamma G}$ can be evaluated row-wise, and this amounts to solving the following problem for each row:
\begin{equation*}
    \begin{aligned}
        &\min_{\mu, \textbf{m}} \frac{1}{2}\|\mu - \mu_{\alpha, 0}\|^{2} + \frac{1}{2}\|\textbf{m} - \textbf{m}_{\alpha, 0}\|^{2},\\
        &\text{ s.t. }\mu - \mu^{k}_{\alpha} + \mathcal{A}\textbf{m} = 0,\quad b_{0, \alpha} \le \mu_{\textbf{i}} \le b_{1, \alpha},\quad \forall \textbf{i} \in \mathfrak{T}.
    \end{aligned}
\end{equation*}
for given $(\mu_{\alpha, 0}, \textbf{m}_{\alpha, 0}) \in \mathfrak{T}(\mathbb{R}) \times \mathfrak{J}(\mathbb{R})$. Closed form solutions to this problem has been given in \cite{PDFB2024}.
\end{remark}

\subsubsection{Proximal operator of nonlinear preconditioner}
The proximal problem associated with $\text{prox}_{\tau\mathcal{U}^{\ast}[\boldsymbol{\nu}/\tau]}$ amounts to solve the following smooth convex optimization problem
\begin{equation*}
    \min_{\boldsymbol{\nu}} \tau\mathcal{U}^{\ast}[\boldsymbol{\nu}/\tau] + \frac{1}{2\overline{\gamma}}\|\boldsymbol{\nu} - \boldsymbol{\nu}_{0}\|^{2}
\end{equation*}
for given $\boldsymbol{\nu}_{0}$. However, $\mathcal{U}^{\ast}$ is not always easy to compute explicitly for given $\mathcal{U}$. In this case, one can turn to Moreau’s identity (see, for example, \cite{Chambolle2011}):
\begin{equation*}
    \text{prox}_{\overline{\gamma}\tau \mathcal{U}^{\ast}[\boldsymbol{\nu}/\tau]}(\boldsymbol{\nu}_{0}) = \boldsymbol{\nu}_{0} - \overline{\gamma}\text{prox}_{\overline{\gamma}^{-1}\tau\mathcal{U}[\boldsymbol{\nu}]}[\boldsymbol{\nu}_{0}/\overline{\gamma}].
\end{equation*}
Alternatively, the proximal problem $\text{prox}_{\overline{\gamma}^{-1}\tau\mathcal{U}[\boldsymbol{\nu}]}[\boldsymbol{\nu}_{0}/\overline{\gamma}]$ amounts to solve the following smooth convex optimization problem
\begin{equation*}
    \min_{\boldsymbol{\nu}} \tau\mathcal{U}[\boldsymbol{\nu}] + \frac{\overline{\gamma}}{2}\|\boldsymbol{\nu} - \boldsymbol{\nu}_{0}/\overline{\gamma}\|^{2}.
\end{equation*}
In either cases, the proximal problem can be solved by Newton's method for smooth optimization
\begin{equation*}
    \min_{\boldsymbol{\nu}} F_{0}(\boldsymbol{\nu}) \Longrightarrow \boldsymbol{\nu}^{\ell + 1} = (\nabla^{2}F_{0}(\boldsymbol{\nu}^{\ell}))^{-1} (\boldsymbol{\nu}^{\ell} - \nabla F_{0}(\boldsymbol{\nu}^{\ell})).
\end{equation*}

\subsubsection{Proximal operator of dual variable}
The proximal problem associated with $\iota_{K}\big(Q_{0}, \boldsymbol{q}_{0}\big)$ amounts to solve the following projection problem for given $(Q_{0}, \boldsymbol{q}_{0}) \in \mathbb{R}^{n \times n}_{\text{sym}} \times \mathbb{R}^{n \times d}$
\begin{equation}\label{Dual-p}
    \min_{Q, \boldsymbol{q}} \frac{1}{2}\big\|Q - Q_{0}\big\|_{2}^{2} + \frac{1}{2}\big\|\boldsymbol{q} - \boldsymbol{q}_{0}\big\|_{2}^{2}\quad\text{ s.t. }Q + \frac{\boldsymbol{q}\boldsymbol{q}^{T}}{2} \preceq_{\mathcal{S}_{+}} 0,
\end{equation}
where $\|\boldsymbol{q}\|_{2} = (\boldsymbol{q}: \boldsymbol{q})^{\frac{1}{2}}$ denotes the Frobenuous norm. This matrix projection problem involves the control $Q$ of eigenvalue of matrices under a quadratic nonlinearity of $\boldsymbol{q}$. We introduce two possible methods for computing~\eqref{Dual-p}.

\vspace{6pt}

\noindent
1. A formalize approach is the alternative direction method of multipliers (ADMM). The Lagrangian of~\eqref{Dual-p} is given as follows,
\begin{equation*}
    \min_{Q, \boldsymbol{q}} \max_{0 \preceq_{\mathcal{S}_{+}} Z} \frac{1}{2}\big\|Q - Q_{0}\big\|^{2}_{2} + \frac{1}{2}\big\|\boldsymbol{q} - \boldsymbol{q}_{0}\big\|^{2}_{2} + Z : \bigg(Q + \frac{\boldsymbol{q} \boldsymbol{q}^{T}}{2}\bigg).
\end{equation*}
where $Z$ is the Lagrangian multiplier. We refer readers to \cite[Example 2.24]{Boyd2004} for the introduction of this Lagrangian multiplier. In each iterate $(Q^{\ell}, \boldsymbol{q}^{\ell}, Z^{\ell})$, we solve
\begin{equation*}
    \begin{aligned}
        Q^{\ell + 1} &= Q_{0} - Z^{\ell} \text{ and } \boldsymbol{q}^{\ell + 1} = (I + Z^{\ell})^{-1}\boldsymbol{q}_{0}\\
        Z^{\ell + 1} &= \text{proj}_{\mathcal{S}_{+}}(Z^{\ell} + \gamma R^{\ell + 1}) \text{ where } R^{\ell + 1} = Q^{\ell + 1} + \frac{\boldsymbol{q}^{\ell + 1} \big(\boldsymbol{q}^{\ell + 1}\big)^{T}}{2}
    \end{aligned}
\end{equation*}
where $\gamma > 0$ is the stepsize parameter, and $\text{proj}_{\mathcal{S}_{+}}$ is the projection onto the space of positive semi-definite matrices. Specifically, to conduct the projection $\text{proj}_{\mathcal{S}_{+}}$ for given $Z_{0}$, we first implement an eigenvalue decomposition, and then truncate its eigenvalues to non-negative real numbers:
\begin{equation*}
    \begin{aligned}
        Z_{0} = U_{Z} \,\text{diag}[\lambda_{\alpha}]_{\alpha = 1}^{n} U^{T}_{Z}, \quad \text{proj}_{\mathcal{S}_{+}}(Z_{0}) = U_{Z} \,\text{diag}[\max(0, \lambda_{\alpha})]_{\alpha = 1}^{n} U^{T}_{Z}.
    \end{aligned}
\end{equation*}

In our experiments, we find the parameters $\gamma$ within $[\frac{1}{2}, 1]$ perform near the optimal. For problems of the size $(n, d) = (2, 1)$, the violation of constrains in the ADMM iterations converges to error $10^{-12}$ usually within 20 steps when $\gamma = 1$. However, it performs depending on the property of $(Q_{0}, \boldsymbol{q}_{0})$ and the choice of $\gamma$, and would be inefficient for problems with large size, for example $(n, d) = (10, 3)$.

\vspace{6pt}

\noindent
2. We also introduce a Newton method. Let $Y = Q + \frac{1}{2}\boldsymbol{q}\boldsymbol{q}^{T}$. For fixed $\boldsymbol{q}$, one has
\begin{equation*}
    \min_{Q} \Big\{\|Q - Q_{0}\|_{2}^{2} : Q \preceq_{\mathcal{S}_{+}} - \tfrac{1}{2}\boldsymbol{q}\boldsymbol{q}^{T}\Big\} = \min_{Y} \Big\{\|Y - (Q_{0} + \tfrac{1}{2}\boldsymbol{q}\boldsymbol{q}^{T})\|_{2}^{2} : Y \preceq_{\mathcal{S}_{+}} 0\Big\}.
\end{equation*}
The RHS of the above equation has minimal expressed by
\begin{equation*}
    \|Y_{+}(\boldsymbol{q})\|_{2}^{2} \text{ where } Y(\boldsymbol{q}) := (Q_{0} + \tfrac{1}{2}\boldsymbol{q}\boldsymbol{q}^{T})
\end{equation*}
where $Y_{+}(\boldsymbol{q}) = \text{proj}_{\mathcal{S}_{+}}(Y(\boldsymbol{q}))$ is the positive part of the matrix. Hence, problem~\eqref{Dual-p} is equivalent to 
\begin{equation*}
    \min_{\boldsymbol{q}} F_{1}(\boldsymbol{q}) = \frac{1}{2}\|\boldsymbol{q} - \boldsymbol{q}_{0}\|_{2}^{2} + \frac{1}{2}\|Y_{+}(\boldsymbol{q})\|_{2}^{2}.
\end{equation*}
Once we get the optimal $\boldsymbol{q}_{\ast}$, we can recover
\begin{equation*}
    Q_{\ast} = Y_{-}(\boldsymbol{q}_{\ast}) - \frac{1}{2}\boldsymbol{q}_{\ast}\boldsymbol{q}_{\ast}^{T}
\end{equation*}
where $Y_{-} = \text{proj}_{\mathcal{S}_{-}}(Y)$ is the negative part of the matrix. To optimize $F_{1}(\boldsymbol{q})$, we note that $F_{1}(\boldsymbol{q})$ is convex and has a subgradient given by
\begin{equation*}
    \partial F_{1}(\boldsymbol{q}) = \boldsymbol{q} - \boldsymbol{q}_{0} + Y_{+}(\boldsymbol{q}) \boldsymbol{q},
\end{equation*}
Hence, it suffice to find the critical point $\partial F_{1}(\boldsymbol{q}) = 0$. To solve this equation, one can consider the Newton method, where in each iterate we update
\begin{equation}\label{Newton-q}
    \boldsymbol{q}^{\ell + 1} = \boldsymbol{q}^{\ell} - H(\boldsymbol{q}^{\ell})^{-1} \big[\partial F_{1}(\boldsymbol{q}^{\ell})\big].
\end{equation}
Here the Hessian operator $H(\boldsymbol{q})$ is given by product rule as follows,
\begin{equation*}
    H(\boldsymbol{q})[\boldsymbol{\xi}] = \lim_{\varepsilon \to 0} \frac{\partial F_{1}(\boldsymbol{q} + \varepsilon \boldsymbol{\xi}) - \partial F_{1}(\boldsymbol{q})}{\varepsilon} = \boldsymbol{\xi} + Y_{+}(\boldsymbol{q})\boldsymbol{\xi} + D\big(Y_{+}(\boldsymbol{q})\big) [\tfrac{1}{2}(\boldsymbol{\xi} \boldsymbol{q}^{T} + \boldsymbol{q} \boldsymbol{\xi}^{T})]\boldsymbol{q},
\end{equation*}
with the gradient $D\big(Y_{+}(\boldsymbol{q})\big)$ defined by
\begin{equation*}
    D\big(Y_{+}(\boldsymbol{q})\big)[Z] = V \big[\mathcal{A} \circ (V^{T}ZV)\big] V^{T},
\end{equation*}
where $Y_{+}(\boldsymbol{q}) = V\mathrm{diag}(\lambda_{1}, \cdots, \lambda_{n})V^{T}$ and $\circ$ is the Hadamard product, i.e., component wise product, and $\mathcal{A}$ is a matrix depending on the eigenvalue of $Y_{+}(\boldsymbol{q})$ defined by
\begin{equation*}
    \mathcal{A}_{\alpha \beta} =
    \begin{cases}
        \displaystyle \frac{(\lambda_{\alpha})_{+} - (\lambda_{\beta})_{+}}{\lambda_{\alpha} - \lambda_{\beta}} &\text{ if } \lambda_{\alpha} \neq \lambda_{\beta}\\
        \chi_{\{\lambda_{\alpha} > 0\}} &\text{ if } \lambda_{\alpha} = \lambda_{\beta}\\
    \end{cases}
\end{equation*}
where $(\lambda_{\alpha})_{+} = \max(0, \lambda_{\alpha})$, and $\chi_{A}$ is the characteristic function of set $A$. Once we have the information of the derivational derivative $H(\boldsymbol{q})[e_{\alpha \beta}]$ for each basis $e_{\alpha \beta}$ of $\mathbb{R}^{n \times d}$, we can expand $\partial F_{1}(\boldsymbol{q}^{\ell})$ as a linear combination of $e_{\alpha \beta}$, and formulate $H(\boldsymbol{q})$ as a matrix on the coefficients. In this case, the Newton iterate~\eqref{Newton-q} can be evaluated as a linear system of the coefficients of expansion.

In our experiments, error in the Newton iterate converges to error $10^{-12}$ usually within 10 steps.

\begin{remark}
For $n = 2$, the eigenvalue decomposition and the matrix inverse can be evaluated analytically, which greatly reduce the computation cost. In addition, a computer code for eigenvalue decomposition of matrix fields is available in \cite[Figure 7]{Goldluecke2012}.
\end{remark}

\section{Numerical results}\label{sec:4}
In this section, we apply the developed PDFB splitting method to computing several cross-diffusion gradient flows~\eqref{1-1}.
\begin{exam}[\cite{SKT, Sun2018, Carrillo2022}]\rm Consider the following SKT population model for $\boldsymbol{\mu} = (\mu_{1}, \mu_{2})$:
\begin{equation}\label{c1}
    M(\boldsymbol{\mu}) = \left[\begin{array}{cc}
        \mu_{1}(2\mu_{1} + \mu_{2}) & \mu_{1}\mu_{2} \\
        \mu_{2}\mu_{1} & \mu_{2}(\mu_{1} + 2 \mu_{2})
    \end{array}\right] \text{ and } \frac{\delta \mathcal{E}[\boldsymbol{\mu}]}{\delta \boldsymbol{\mu}} = \left[\begin{array}{c}
        \log(\mu_{1}) \\
        \log(\mu_{2})
    \end{array}\right].
\end{equation}
The free energy functional associated with this model is given by
\begin{equation*}
    \mathcal{E}[\boldsymbol{\mu}] = \int_{\Omega} \,\mu_{1}(\log(\mu_{1}) - 1) + \mu_{2}(\log(\mu_{2}) - 1) \,\mathrm{d}x.
\end{equation*}

In this example, we consider both one dimensional and two dimensional models. For $d = 1$, we consider the following initial value
\begin{equation*}
    \mu_{1}(x, 0) = \frac{1}{(2\pi)^{\frac{1}{2}}}\exp\left(\frac{-(x - \tfrac{1}{2})^{2}}{2}\right),\quad \mu_{2}(x, 0) = \frac{1}{(2\pi)^{\frac{1}{2}}}\exp\left(\frac{-(x + \tfrac{1}{2})^{2}}{2}\right).
\end{equation*}
In this case, we solve~\eqref{c1} on $\Omega = [-5, 5]$ with a grid spacing $h = 0.1$ and a time step size $\tau = 0.1$. For $d = 2$, we consider the following initial value
\begin{equation*}
    \begin{aligned}
        \mu_{1}(x, y, 0) &= \frac{1}{(2\pi)^{\frac{1}{2}}}\exp\left(\frac{-x^{2} -(y + \tfrac{1}{2})^{2}}{2}\right),\\
        \mu_{2}(x, y, 0) &= \frac{1}{(2\pi)^{\frac{1}{2}}}\exp\left(\frac{-x^{2} -(y - \tfrac{1}{2})^{2}}{2}\right).
    \end{aligned}
\end{equation*}
In this case, we solve~\eqref{c1} on $\Omega = [-5, 5] \times [-5, 5]$ with a grid spacing $h = 10/256$ and a time step size $\tau = 0.1$. In both cases, to ensure bound preservation $\boldsymbol{\mu} \ge 0$, we set $\mathcal{M} = I_{2 \times 2}$, $\boldsymbol{b}_{0} = [0; 0]$ and $\boldsymbol{b}_{1} = [\infty; \infty]$ in the box constraint.

The discrete free energy of the model is
\begin{equation*}
    \mathcal{E}_{h}[\boldsymbol{\mu}] = \sum_{\textbf{i} \in \mathfrak{T}} \mu_{1, \textbf{i}}(\log(\mu_{1, \textbf{i}}) - 1)h^{d} + \mu_{2, \textbf{i}}(\log(\mu_{2, \textbf{i}}) - 1)h^{d}.
\end{equation*}
As the Gibbs-Boltzmann entropy is ill-conditioned, one can consider the convex splitting approach by taking $\mathcal{U}[\boldsymbol{\mu}] = \widehat{\mathcal{E}}_{h}[\boldsymbol{\mu}]$. The convex conjugate to $\mathcal{U}$ can be explicitly computed as follows,
\begin{equation*}
    \mathcal{U}^{\ast}[\boldsymbol{\nu}] = \sum_{\textbf{i} \in \mathfrak{T}} e^{-\nu_{1, \textbf{i}}} + e^{-\nu_{2, \textbf{i}}} \text{ where } \boldsymbol{\nu} = (\nu_{1}, \nu_{2}).
\end{equation*}

\begin{figure}[h]
    \centering
    \includegraphics[width = \textwidth]{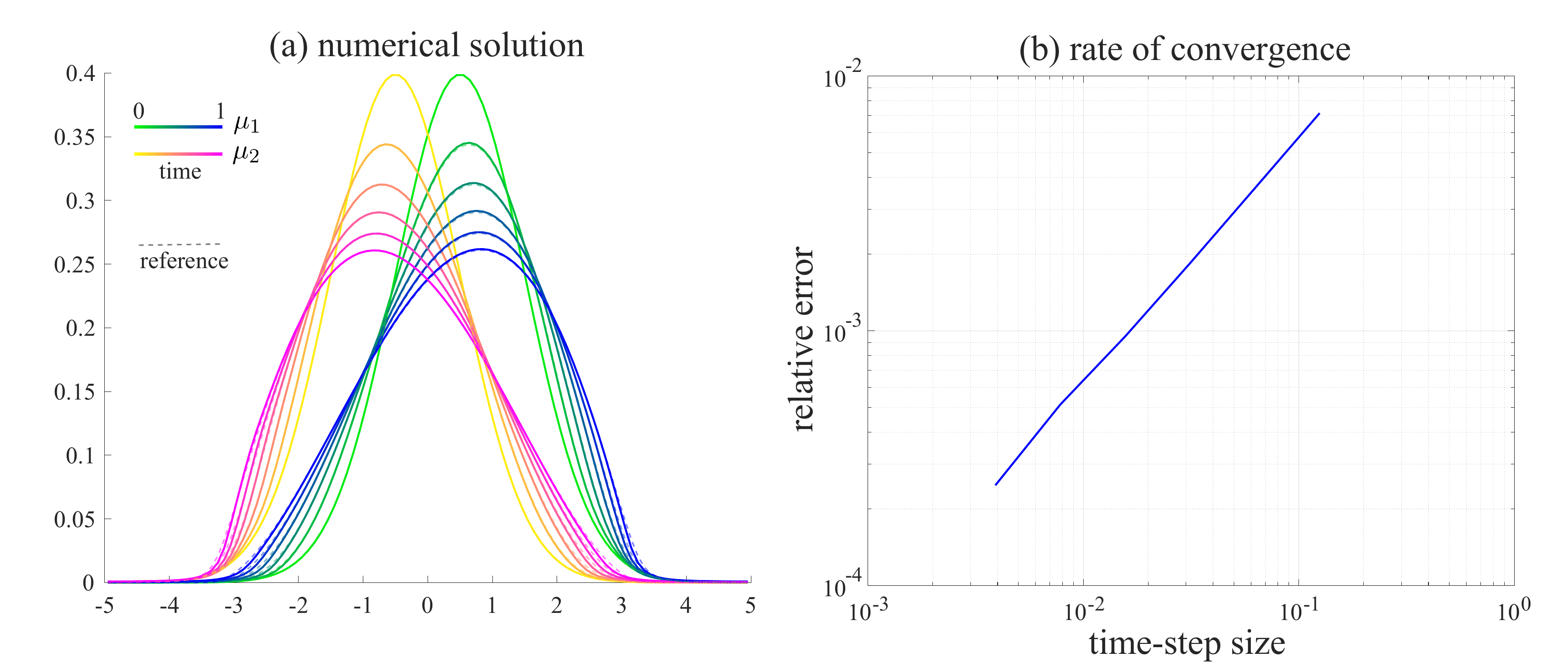}
    \caption{\rm Numerical results in \textit{Example 1} in one dimension. (a) The numerical solution at several time steps, where the dash-lines represent the reference solution. (b) Plot of the relative errors at $t = 1$ with respect to different time step sizes.}
    \label{fig:1-1}
\end{figure}

The numerical results are given in Figure~\ref{fig:1-1} ($d = 1$) and Figure~\ref{fig:1-2} ($d = 2$). We use reference solutions to justify our PDFB splitting method. These reference solutions are obtained by classical backward Euler method with a time step size $\tau = 10^{-3}$ on the same spatial grid with central difference approximation. In Figure~\ref{1-1} (a) and Figure~\ref{fig:1-2} (d, e), we can see that the numerical solutions of the PDFB splitting method closely matches the reference solutions. In addition, to show the accuracy of the time-discrete scheme, in Figure~\ref{1-1} (b), we plot the rate of convergence for the relative error with respect to a reference solution obtained by classical backward Euler method with $\tau = 10^{-4}$, and we adopt a grid spacing $h = 10/1024$ for both of the schemes in the comparison.

\begin{figure}
    \centering
    \includegraphics[width=\linewidth]{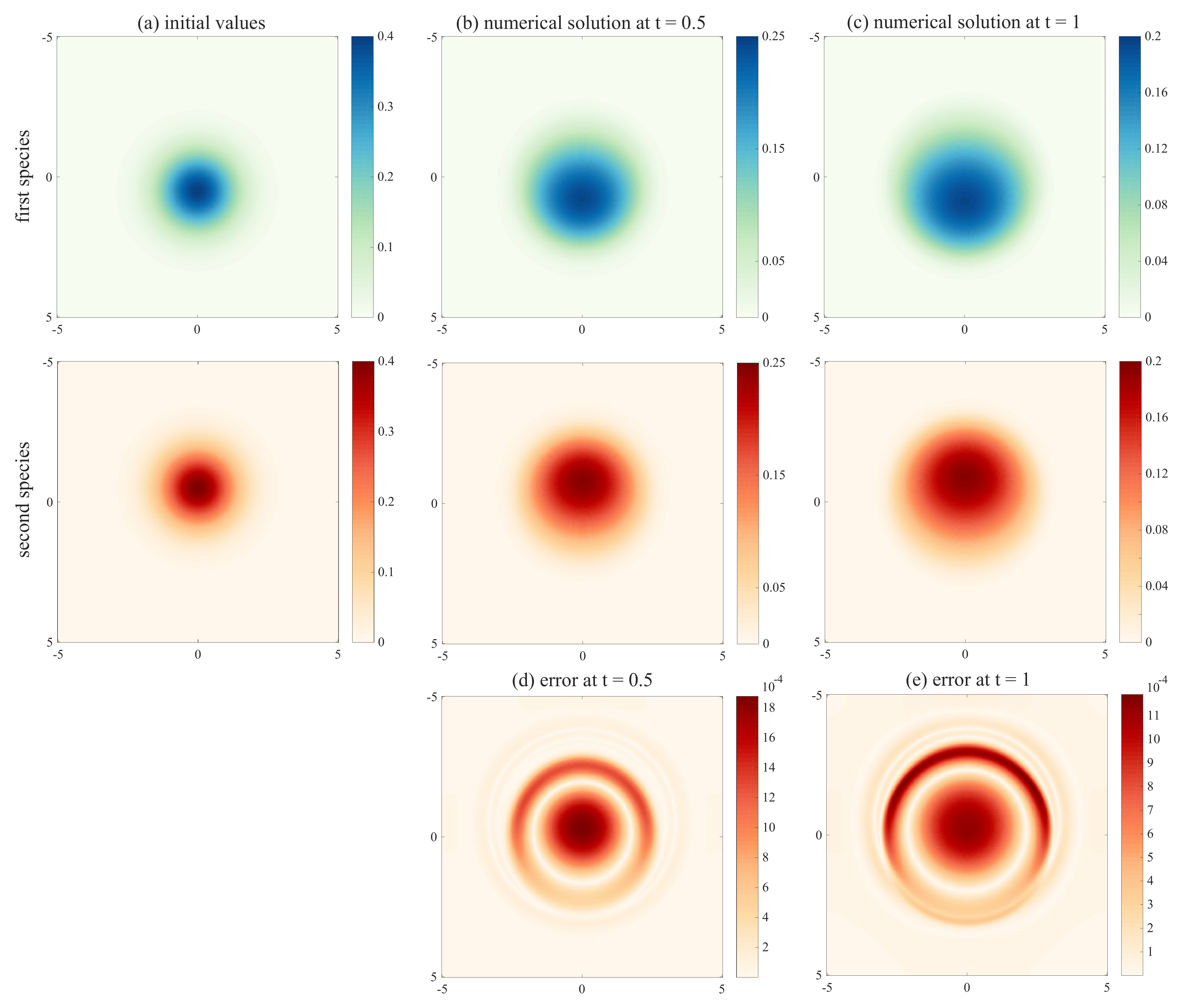}
    \caption{\rm Numerical results in \textit{Example 1} in two dimension. (a) The initial values of the two species, respectively. (b) Numerical solution of the two species at $t = 0.5$, respectively. (c) Numerical solution of the two species at $t = 1$, respectively. (d) The error of the numerical solution of the second species at $t = 0.5$. (e) The error of the numerical solution of the second species at $t = 1$. In (d) and (e), the error is evaluated via $|\mu_{2} - \overline{\mu}_{2}|$ where $\overline{\mu}_{2}$ is the reference solution.}
    \label{fig:1-2}
\end{figure}
\end{exam}

\begin{exam}[\cite{Barrett2003, Barrett2004, Thiele2012}]\rm Consider the following model describing a surfactant spreading on a thin viscous film for $\boldsymbol{\mu} = (\mu_{1}, \mu_{2})$:
\begin{equation}\label{example2}
    \begin{aligned}
        M(\boldsymbol{\mu}) = \left[\begin{array}{cc}
            \mu_{1}^{3}/3 & \mu_{1}^{2}\mu_{2}/2\\
            \mu_{1}^{2}\mu_{2}/2 & \mu_{1}\mu_{2}^{2} + \varepsilon \mu_{2}
        \end{array}\right] \text{ and }
        \frac{\delta \mathcal{E}[\boldsymbol{\mu}]}{\delta \boldsymbol{\mu}} = \left[\begin{array}{c}
            -\Delta \mu_{1}\\
            \log(\mu_{2})
        \end{array}\right],
    \end{aligned}
\end{equation}
where $\mu_{1}$ represents the thickness of the film, $\mu_{2}$ is the concentration of the surfactant. The free energy functional associated with this example is
\begin{equation*}
    \mathcal{E}[\boldsymbol{\mu}] = \int_{\Omega} \frac{1}{2}|\nabla \mu_{1}|^{2} + \mu_{2}\big(\log(\mu_{2}) - 1\big)\,\mathrm{d}x.
\end{equation*}

In this example, we use the following initial value considered \cite{Barrett2003},
\begin{equation*}
    \mu_{1}(x, 0) = 1,\quad \mu_{2}(x, 0) = \frac{1}{2}(1 - \tanh(10|x| - 5)).
\end{equation*}
Here we solve~\eqref{example2} on $\Omega = [-4, 4]$ with two different grid spacings, a coarser grid where $h = 1/128$ and a finer grid where $h = 1/512$, and with a time step size $\tau = 0.1$. To ensure bound preservation $\mu_{2} \ge 0$, we set $\mathcal{M} = I_{2 \times 2}$, $\boldsymbol{b}_{0} = [-\infty; 0]$ and $\boldsymbol{b}_{1} = [\infty; \infty]$ in the box constraint.
\begin{figure}[h]
    \centering
    \includegraphics[width = \textwidth]{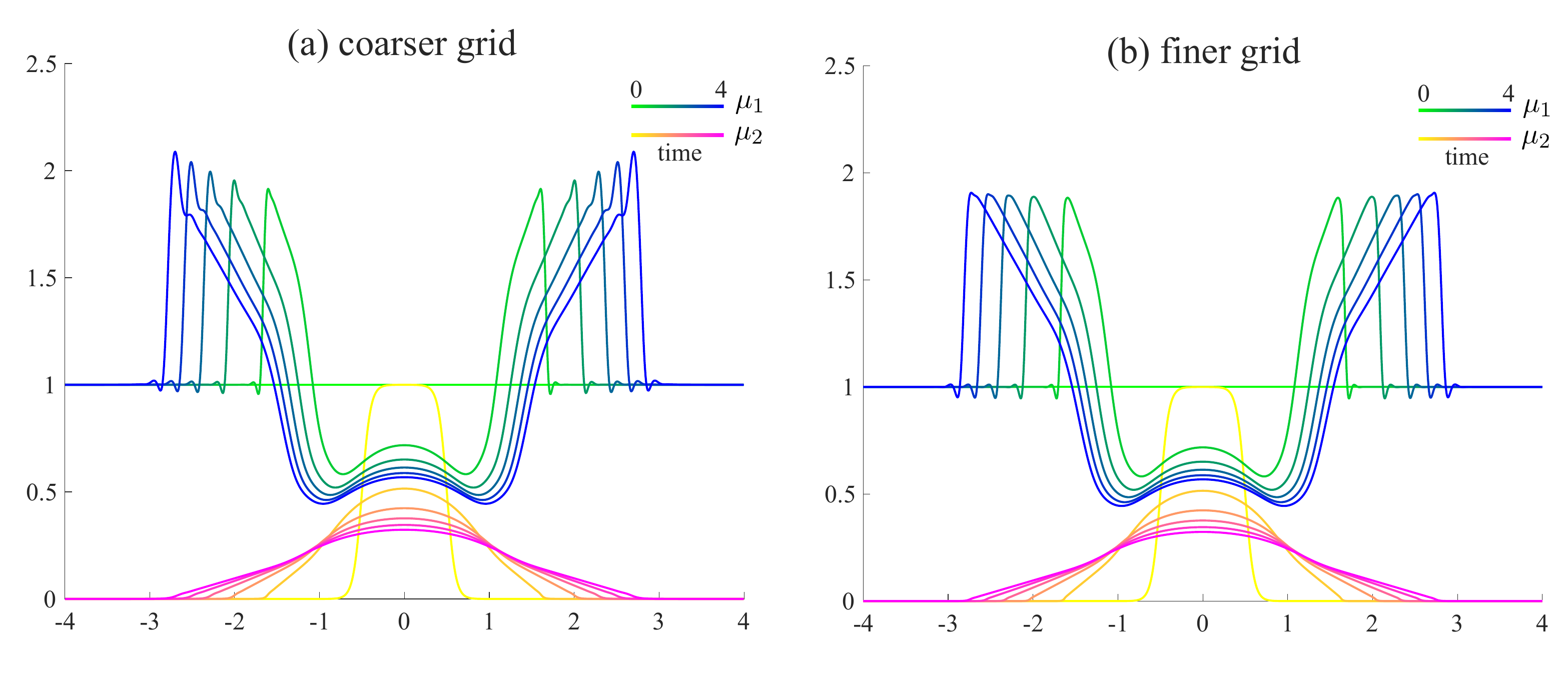}
    \caption{\rm Numerical results in \textit{Example 2}. (a) Numerical solution at several time steps obtained with a coarser grid spacing ($h = 1/128$). (b) Numerical solution at several time steps obtained with a finer grid spacing ($h = 1/512$).}\label{fig:5-2}
\end{figure}

The discrete free energy of the model in this example is
\begin{equation*}
    \mathcal{E}_{h}[\boldsymbol{\mu}] = \sum_{\textbf{j} \in \mathfrak{J}^{1}} \frac{1}{2} |\partial_{1, h} \mu_{1}|^{2}_{\textbf{j}} h + \sum_{\textbf{i} \in \mathfrak{T}} \mu_{2, \textbf{i}}(\log(\mu_{2, \textbf{i}}) - 1)h,
\end{equation*}
We consider the convex splitting though taking $\mathcal{U}[\boldsymbol{\mu}] = \widehat{\mathcal{E}}_{h}[\boldsymbol{\mu}]$. As the proximal problem associated with $\mathcal{U}[\boldsymbol{\mu}]$ is separable in $\mu_{1}$ and $\mu_{2}$, we handle it as two individual subproblems. Herein, the proximal problem of the Dirichlet energy can be evaluated using Moreau’s identity via solving the following quadratic optimization problem (for given $\nu_{1, 0}$):
\begin{equation*}
    \min_{\nu} \tau \sum_{\textbf{j} \in \mathfrak{J}^{1}} \frac{1}{2} |\partial_{1, h} \nu|^{2}_{\textbf{j}} + \frac{\overline{\gamma}}{2}\|\nu - \nu_{1, 0}/\overline{\gamma}\|^{2} \Longrightarrow \nu_{1} = (\overline{\gamma} I - \tau \Delta_{h})^{-1}\nu_{1, 0}
\end{equation*}
where $\Delta_{h}$ is the centered approximation of Laplacian on uniform grids. Furthermore, the proximal operator corresponding to $\sum_{\textbf{i} \in \mathfrak{T}} \mu_{2, \textbf{i}} \big(\log(\mu_{2, \textbf{i}}) - 1\big)$ can be solved the same way as \textit{Example 1}.

The numerical results are given in Figure~\ref{fig:5-2}. We observe a similar behavior of the numerical solution (Figure~\ref{fig:5-2} (a)) on the top of the shock as that in \cite[Figure 4.3]{Sun2018} when the spatial resolution is low, which can be addressed via increasing the spatial resolution. A reference to Figure~\ref{fig:5-2} (b) is given in \cite[Figure 4]{Barrett2003}.
\end{exam}

\begin{exam}[\cite{Jachalski2012}]\rm Consider the liquid two-layer thin film model for $\boldsymbol{\mu} = (\mu_{1}, \mu_{2})$
\begin{equation}\label{model3}
    \begin{aligned}
        &M(\boldsymbol{\mu}) = \left[\begin{array}{cc}
            \mu_{1}^{3}/3 & \mu_{1}^{3}/3 + \mu_{1}^{2}(\mu_{2} - \mu_{1})/2 \\
            \mu_{1}^{3}/3 + \mu_{1}^{2}(\mu_{2} - \mu_{1})/2 & (\mu_{2} - \mu_{1})^{3}/3 + \mu_{1}\mu_{2}(\mu_{2} - \mu_{1}) + \mu_{1}^{3}/3
        \end{array}\right]\\
        &\text{ and } \frac{\delta \mathcal{E}[\boldsymbol{\mu}]}{\delta \boldsymbol{\mu}} = \left[\begin{array}{c}
            -\nu \Delta \mu_{1} - P_{\varepsilon}^{\prime}(\mu_{2} - \mu_{1}) \\
            -\Delta \mu_{2} + P_{\varepsilon}^{\prime}(\mu_{2} - \mu_{1})
        \end{array}\right],
    \end{aligned}
\end{equation}
where $\mu_{1}, \mu_{2}$ represents the thickness of different layers, $\nu$ is the ratio of viscosities, and $P_{\varepsilon}$ represents a intermolecular potential between the two liquids. In this example, we consider the Lennard-Jones potential \cite{Oron1997} given by
\begin{equation*}
    P_{\varepsilon}(r) = \frac{\varepsilon^{8}}{8r^{8}} - \frac{\varepsilon^{2}}{2r^{2}}.
\end{equation*}
The free energy functional associated with this model is given by
\begin{equation*}
    \mathcal{E}(\boldsymbol{\mu}) = \int_{\Omega} \frac{\sigma}{2}|\nabla \mu_{1}|^{2} + \frac{1}{2}|\nabla \mu_{2}|^{2} + P_{\varepsilon}(\mu_{2} - \mu_{1}) \,\mathrm{d}x.
\end{equation*}

In this example, we set $\varepsilon = 0.01$, and use the following initial value,
\begin{equation*}
    \mu_{1}(x, 0) = \frac{3}{4} - \frac{1}{4}\cos\left(\frac{\pi x}{2}\right) \text{ and } \mu_{2}(x, 0) = 1 + \varepsilon.
\end{equation*}
Here we solve~\eqref{model3} on $\Omega = [-1, 1]$ with a grid spacing $h = 0.02$ and a time step size $\tau = 0.0001$. The box constraint is not considered in this example.

\begin{figure}[h!]
    \centering
    \includegraphics[width = \textwidth]{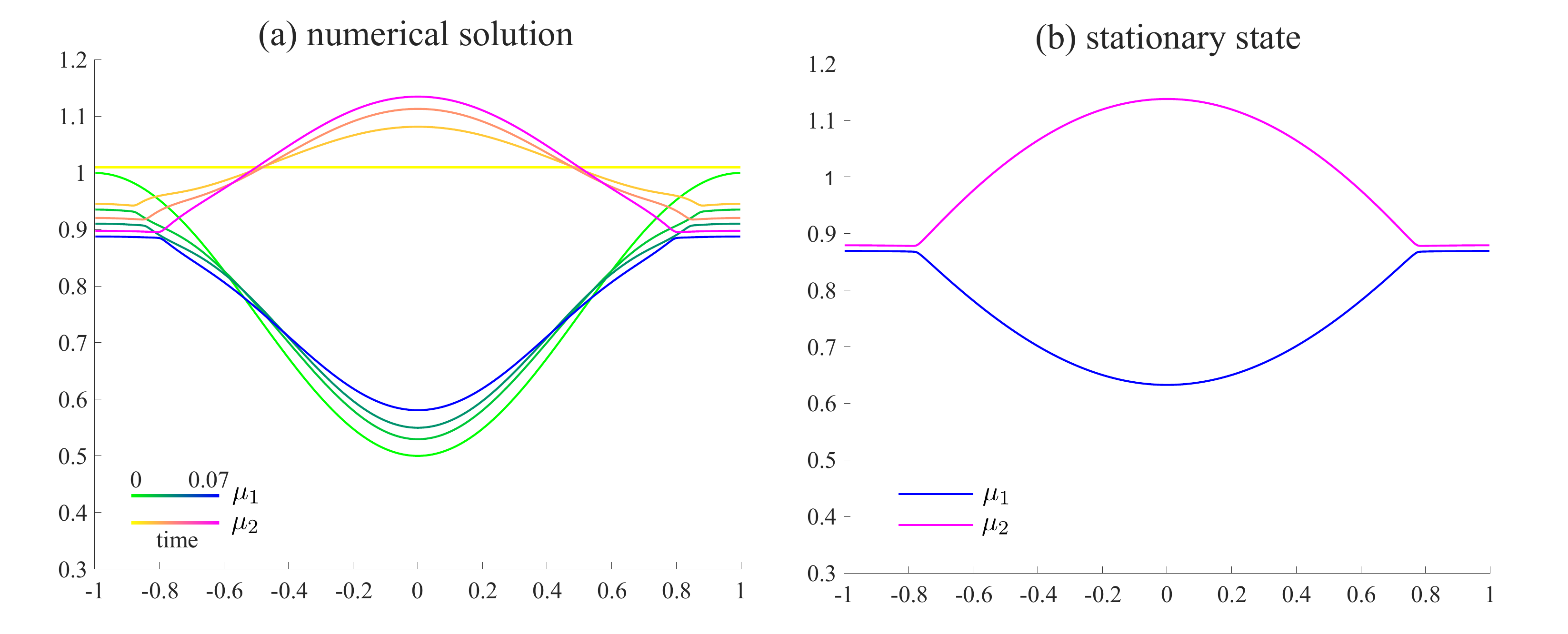}
    \caption{\rm Numerical results in \textit{Example 3}. (a) The numerical solution at several time steps. (b) The numerical solution computed at $t = 0.5$ which can be seen as an approximation the stationary state of the problem.}\label{fig:5-1}
\end{figure}

The discrete free energy of the model in this example is
\begin{equation*}
    \mathcal{E}_{h}[\boldsymbol{\mu}] = \sum_{\textbf{j} \in \mathfrak{J}^{1}} \frac{\sigma}{2} |\partial_{1, h} \mu_{1}|^{2}_{\textbf{j}} h + \frac{1}{2} |\partial_{1, h} \mu_{2}|^{2}_{\textbf{j}} h + \sum_{\textbf{i} \in \mathfrak{T}} P_{\varepsilon}(\mu_{2, \textbf{i}} - \mu_{1, \textbf{i}})h,
\end{equation*}
Since the Lennard-Jones potential $P_{\varepsilon}(\nu)$ is singular around $\nu \approx 0$, we consider the convex splitting via taking
\begin{equation*}
    \mathcal{U}[\mathcal{K} \boldsymbol{\mu}] = \sum_{\textbf{i} \in \mathfrak{T}}P_{\varepsilon}(\mu_{2, \textbf{i}} - \mu_{1, \textbf{i}}) \text{ and }
    \frac{\delta \mathcal{V}[\boldsymbol{\mu}]}{\delta \boldsymbol{\mu}} = \left[\begin{array}{c}
            -\sigma \Delta_{h} \mu_{1}\\
            -\Delta_{h} \mu_{2}
        \end{array}\right],
\end{equation*}
where $\mathcal{K}\boldsymbol{\mu} = \mu_{2} - \mu_{1}$. Herein, the proximal problem of the Lennard-Jones potential can be evaluated using the Moreau’s identity via solving the following optimization
\begin{equation*}
    \min_{\nu} \tau \sum_{\textbf{i} \in \mathfrak{T}} P_{\varepsilon}(\nu_{\textbf{i}}) + \frac{\overline{\gamma}}{2}\|\nu - \nu_{0}/\overline{\gamma}\|^{2},
\end{equation*}
which can be solved by several Newton iterates.

The numerical results are given in Figure~\ref{fig:5-1}. For large-time, we can see that the numerical solution converges to a stationary state Figure~\ref{fig:5-1} (b) that resembles the theoretical stationary state \cite[Fig 4.1]{Jachalski2012}.
\end{exam}

\begin{figure}[t!]
    \centering
    \includegraphics[width=0.9\linewidth]{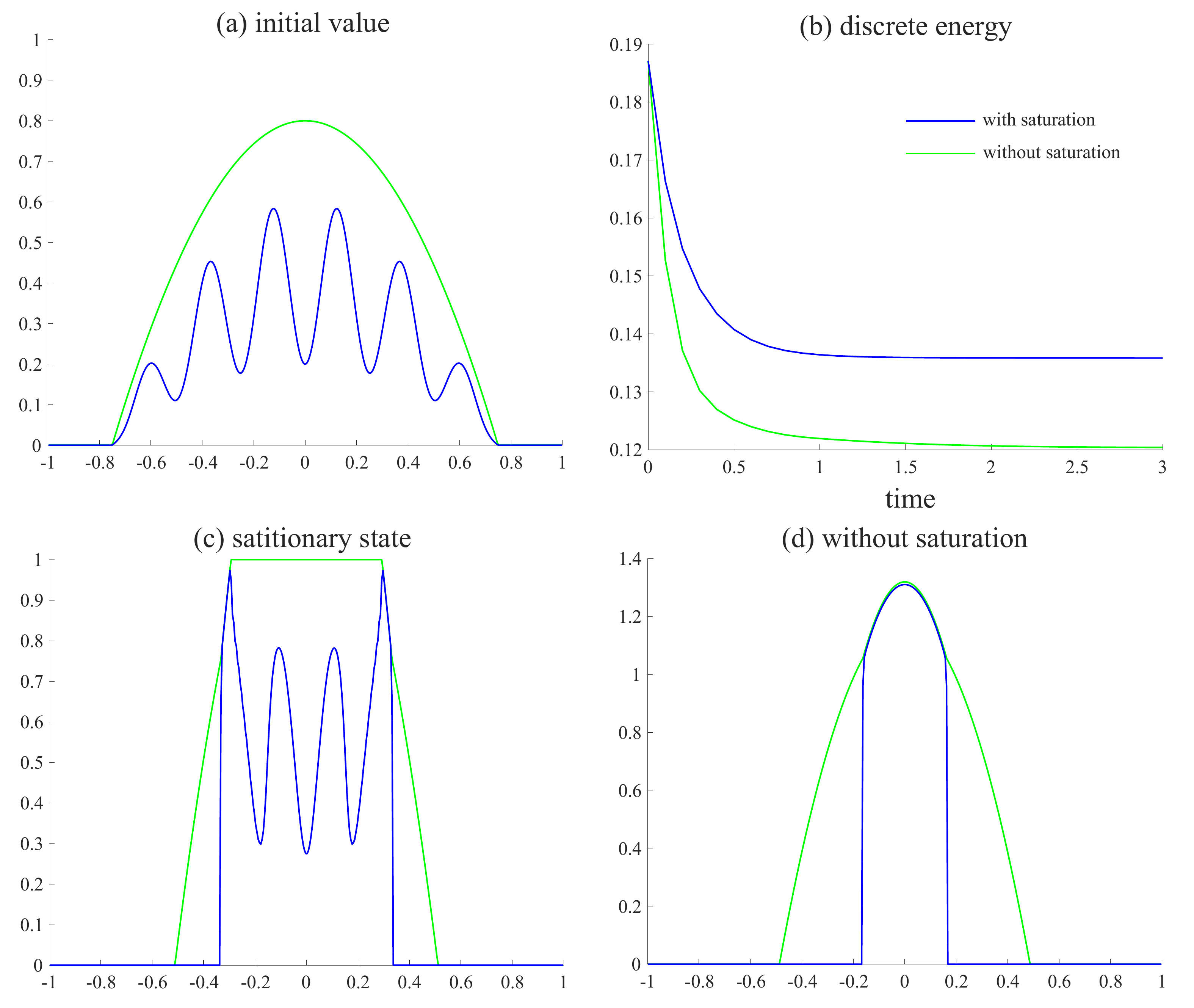}
    \caption{\rm Numerical results in \textit{Example 4}. (a) The plot of the initial value. (b) The evolution of the discrete free energies. (c) Stationary state of the problem (with saturation effect). (d) Stationary state of the problem without saturation, i.e., the mobility $M(\boldsymbol{\mu}) = \text{diag}(\boldsymbol{\mu})$. In (a), (c), and (d), the blue line represents the numerical solution of $\mu_{1}$, and the green color represents the line represents the saturation level, i.e., $\mu_{1} + \mu_{2}$.}
    \label{fig:5-5}
\end{figure}

\begin{exam}[\cite{Carrillo2022}]\rm Consider the following system of Fokker-Planck equations with a saturation effect for $\boldsymbol{\mu} = (\mu_{1}, \mu_{2})$
\begin{equation}\label{model4}
    \begin{aligned}
        M(\boldsymbol{\mu}) &= \left[\begin{array}{cc}
        \mu_{1}(1 - (\mu_{1} + \mu_{2})) & 0 \\
        0 & \mu_{2}(1 - (\mu_{1} + \mu_{2}))
    \end{array}\right]\\
    &\text{ and } \frac{\delta \mathcal{E}[\boldsymbol{\mu}]}{\delta \boldsymbol{\mu}} = \left[\begin{array}{c}
        a (\mu_{1} + \mu_{2}) + \sigma_{1}|x|^{2}/2 \\
        a (\mu_{1} + \mu_{2}) + \sigma_{2}|x|^{2}/2
    \end{array}\right].
    \end{aligned}
\end{equation}
The free energy functional associated with this model is given by
\begin{equation*}
    \mathcal{E}[\boldsymbol{\mu}] = \int_{\Omega} \frac{a}{2}(\mu_{1} + \mu_{2})^{2} + \frac{\sigma_{1}}{2}|x|^{2} \mu_{1} + \frac{\sigma_{2}}{2} |x|^{2} \mu_{2} \, \mathrm{d}x.
\end{equation*}

In this example, we set $a = 0.2$, $\sigma_{1} = 4$ and $\sigma_{2} = 2$, and use the following initial value,
\begin{equation*}
    \mu_{1}(x, 0) = \Big[f(x) \big(1 - \tfrac{1}{2}\cos(\omega x)\big)\Big]_{+} \text{ and } \mu_{2}(x, 0) = \Big[f(x) \big(1 + \tfrac{1}{2}\cos(\omega x)\big)\Big]_{+}.
\end{equation*}
where $f(x) = \frac{2}{5}\big(1 - (\frac{4}{3}x)^{2}\big)$, and $\omega = 8\pi$. Here we solve~\eqref{model4} on $\Omega = [-1, 1]$ with a grid spacing $h = 0.005$ and a time step size $\tau = 0.1$. To ensure bound preservation such that $M(\boldsymbol{\mu}) \ge 0$, we set the parameters in the box constraint as that in \textit{Remark~\ref{box5}}.

The discrete free energy of the model in this example is
\begin{equation*}
    \mathcal{E}_{h}[\boldsymbol{\mu}] = \sum_{\textbf{i} \in \mathfrak{T}} \frac{a}{2}(\mu_{1, \textbf{i}} + \mu_{2, \textbf{i}})^{2} h + \frac{\sigma_{1}}{2}|x_{\textbf{i}}|^{2} \mu_{1, \textbf{i}} h + \frac{\sigma_{2}}{2} |x_{\textbf{i}}|^{2} \mu_{2, \textbf{i}} h.
\end{equation*}
The convex splitting approach is not considered in this example.

The numerical results are given in Figure~\ref{fig:5-5}. We can see that the saturation $\mu_{1} + \mu_{2}$ are confined in the region $[0, 1]$ at the discrete level. %In addition, we observe that the stationary numerical solution is close to the stationary state in \citep[Fig 6]{Carrillo2022}.
\end{exam}

\section{Proof of the support function form}\label{sec:supp}
In this part, we prove the support function characterization in Theoerm~\ref{dual_lemma}. We let $f_{\text{supp}}(M, \textbf{m})$ be the support function form defined by
\begin{equation*}
    f_{\text{supp}}(M, \textbf{m}) = \max_{(Q, \boldsymbol{q}) \in K} M : Q + \textnormal{\textbf{m}} : \boldsymbol{q}.
\end{equation*}
In order to prove Theorem~\ref{dual_lemma}, we first show the fact that
\begin{equation}\label{sup_ineq}
    f_{\text{supp}}(M, \textbf{m}) \le f(M, \textbf{m}).
\end{equation}
We then find $(Q_{\ast}, \boldsymbol{q}_{\ast}) \in K$ such that 
\begin{equation}\label{sup_eqq}
    M : Q_{\ast} + \textnormal{\textbf{m}} : \boldsymbol{q}_{\ast} = f(M, \textbf{m}),
\end{equation}
for admissible pairs $(M, \textbf{m})$, and $f_{\text{supp}}(M, \textbf{m}) = \infty$ for non-admissible pairs $(M, \textbf{m})$.

\begin{proof}[Proof of Inequality~\eqref{sup_ineq}]
Consider the eigenvalue decomposition
\begin{equation*}
    M = U^{T} \Lambda U
\end{equation*}
where $U$ is an orthonormal matrix, and $\Lambda = \text{diag}[\lambda_{\alpha}]_{\alpha = 1}^{n}$ is a diagonal matrix. Applying this eigenvalue decomposition to $M$ in $f(M, \textbf{m})$, we can show that
\begin{equation}\label{decom-f}
    f(M, \textbf{m}) = \sum_{\alpha = 1}^{n}f_{0}(\lambda_{\alpha}, [U\textbf{m}]_{\alpha}),
\end{equation}
where $[U\textbf{m}]_{\alpha}$ is the $\alpha$-th row of the matrix $U\textbf{m}$, and the function $f_{0}$ is defined by,
\begin{equation*}
    f_{0}\big(\lambda_{\alpha}, \textbf{m}_{\alpha}\big) =
    \begin{cases}
        \displaystyle\frac{|\textbf{m}_{\alpha}|^{2}}{2\lambda_{\alpha}} & \lambda_{\alpha} > 0,\\
        0 & \big(\lambda_{\alpha}, \textbf{m}_{\alpha}\big) = (0, 0),\\
        \infty & \text{ otherwise. }\\
    \end{cases}
\end{equation*}
The structure of $f_{0}$ is well explored for scalar gradient flows \cite{Carrillo2019, Carrillo2023, PDFB2024}. In particular, we have the following support function form \cite{Benamou2000}:
\begin{equation*}
    f_{0}\big(\lambda_{\alpha}, [U\textbf{m}]_{\alpha}\big) = \max_{\overline{Q}_{\alpha}, \overline{\boldsymbol{q}}_{\alpha}} \lambda_{\alpha} \overline{Q}_{\alpha} + [U\textbf{m}]_{\alpha} \cdot \overline{\boldsymbol{q}}_{\alpha}\quad \text{ s.t. } \overline{Q}_{\alpha} + \frac{|\overline{\boldsymbol{q}}_{\alpha}|^{2}}{2} \le 0.
\end{equation*}
Then using this formula, we can see that
\begin{equation}\label{supp1}
    f(M, \textbf{m}) = \sum_{\alpha = 1}^{n} \max_{\overline{Q}_{\alpha}, \overline{\boldsymbol{q}}_{\alpha}} \lambda_{\alpha} \overline{Q}_{\alpha} + [U\textbf{m}]_{\alpha} \cdot \overline{\boldsymbol{q}}_{\alpha}\quad \text{ s.t. } \overline{Q}_{\alpha} + \frac{|\overline{\boldsymbol{q}}_{\alpha}|^{2}}{2} \le 0\quad \forall \alpha.
\end{equation}
However, such a support function characterization depends on the form of eigenvalue decomposition of the mobility $M$. To develop a support function as that in Theorem~\ref{dual_lemma}, we introduce the following matrix version of~\eqref{supp1},
\begin{equation*}
    f(M, \textbf{m}) = \max_{\overline{Q}, \overline{\boldsymbol{q}}}\Lambda : \overline{Q} + U\textbf{m} : \overline{\boldsymbol{q}}\quad \text{ s.t. } (\overline{Q}, \overline{\boldsymbol{q}}) \in \mathbb{R}_{\text{sym}}^{n \times n} \times \mathbb{R}^{n \times d} \text{ and } \overline{Q}_{\alpha} + \frac{|\overline{\boldsymbol{q}}_{\alpha}|^{2}}{2} \le 0\quad \forall \alpha.
\end{equation*}
where $\overline{Q}_{\alpha}$ is the $\alpha$-th diagonal element of the symmetrical matrix $\overline{Q}$, and $\overline{\boldsymbol{q}}_{\alpha}$ is the $\alpha$-th row of the matrix $\overline{\boldsymbol{q}}$. This support function form holds true because the matrix $\Lambda$ is zero outside the diagonal. Besides, the constraint $\overline{Q}_{\alpha} + \frac{1}{2}|\overline{\boldsymbol{q}}_{\alpha}|^{2} \le 0$ in~\eqref{supp1} is equivalent to the diagonal elements of the matrix $\overline{Q} + \frac{1}{2}\overline{\boldsymbol{q}}\overline{\boldsymbol{q}}^{T}$ being non-positive.

We then consider the change of variable $Q = U^{T} \overline{Q} U$ and $\boldsymbol{q} = U^{T}\overline{\boldsymbol{q}}$. By using the fact that $U$ is an orthonormal matrix, we have
\begin{equation}\label{duality}
    \begin{aligned}
        \Lambda : \overline{Q} + U\textbf{m} : \overline{\boldsymbol{q}} &= U^{T} \Lambda U : U^{T} \overline{Q} U + \textbf{m} : U^{T}\overline{\boldsymbol{q}}\\
        &= M : Q + \textbf{m} : \boldsymbol{q}.
    \end{aligned}
\end{equation}
We also note that
\begin{equation*}
    U\bigg(Q + \frac{\boldsymbol{q} \boldsymbol{q}^{T}}{2}\bigg)U^{T} = \overline{Q} + \frac{\overline{\boldsymbol{q}}\overline{\boldsymbol{q}}^{T}}{2}.
\end{equation*}
Hence, from~\eqref{duality} and by using the fact that
\begin{equation*}
    U_{\alpha}\bigg(Q + \frac{\boldsymbol{q} \boldsymbol{q}^{T}}{2}\bigg)U_{\alpha}^{T} = e_{\alpha}^{T}\bigg(\overline{Q} + \frac{\overline{\boldsymbol{q}}\overline{\boldsymbol{q}}^{T}}{2}\bigg)e_{\alpha} = \overline{Q}_{\alpha} + \frac{|\overline{\boldsymbol{q}}_{\alpha}|^{2}}{2} \le 0,
\end{equation*}
where $e_{\alpha}$ is the $\alpha$-th vector of the standard basis of $\mathbb{R}^{n}$ and $U_{\alpha} = e_{\alpha}^{T}U$, one can show that the action function $f$ can be characterized using the following support function form.
\begin{equation}\label{duality2}
    \begin{aligned}
        f(M, \textbf{m}) &= \max_{Q, \boldsymbol{q}} M : Q + \textbf{m} : \boldsymbol{q}\\
        &\text{ s.t. } (Q, \boldsymbol{q}) \in \mathbb{R}_{\text{sym}}^{n \times n} \times \mathbb{R}^{n \times d} \text{ and } U_{\alpha}\bigg(Q + \frac{\boldsymbol{q} \boldsymbol{q}^{T}}{2}\bigg)U_{\alpha}^{T} \le 0 \quad\forall \alpha.
    \end{aligned}
\end{equation}

Finally, we note that the admissible dual variables $(Q, \boldsymbol{q})$ are characterized by $Q + \frac{1}{2}\boldsymbol{q} \boldsymbol{q}^{T} \preceq_{\mathcal{S}_{+}} 0$ in $f_{\text{supp}}(M, \textbf{m})$. By using the variational characterization of eigenvalues, we have $U^{T}_{\ast}(Q + \frac{1}{2}\boldsymbol{q} \boldsymbol{q}^{T})U_{\ast} \le 0$ for all $U_{\ast} \in \mathbb{R}^{n}$. Hence, the condition $Q + \frac{1}{2}\boldsymbol{q} \boldsymbol{q}^{T} \preceq_{\mathcal{S}_{+}} 0$ is sufficient for the constraint in~\eqref{duality2}, which gives us
\begin{equation*}
    f(M, \textbf{m}) \ge f_{\text{supp}}(M, \textbf{m}) = \max_{Q, \boldsymbol{q}} M : Q + \textbf{m} : \boldsymbol{q} \quad \text{ s.t. } Q + \frac{\boldsymbol{q} \boldsymbol{q}^{T}}{2} \preceq_{\mathcal{S}_{+}} 0,
\end{equation*}
as maximization of the same objective function over a smaller set results in a smaller maximum.
\end{proof}

\begin{proof}[Construction of the recovery sequence~\eqref{sup_eqq}]~

\noindent
1. In the \textit{admissible} case where $M$ is positive semi-definite and invertible, i.e., $0 \preceq_{\mathcal{S}_{+}} M$ and there exists $\textbf{u}$ such that $\textbf{m} = M\textbf{u}$, we let
\begin{equation*}
    Q_{\ast} = -\frac{1}{2} M^{\dagger} \textbf{m} \textbf{m}^{T}M^{\dagger}\quad \boldsymbol{q}_{\ast} = M^{\dagger} \textbf{m}.
\end{equation*}
It is direct to show that $Q_{\ast} + \frac{1}{2}\boldsymbol{q}_{\ast} \boldsymbol{q}_{\ast}^{T} = 0$, and therefore $(Q_{\ast}, \boldsymbol{q}_{\ast}) \in K$. In addition, we have
\begin{equation*}
    M : Q_{\ast} + \textbf{m} : \boldsymbol{q}_{\ast} = \frac{1}{2}\textbf{m} : M^{\dagger} \textbf{m} = f(M, \textbf{m}).
\end{equation*}
which means that $f_{\text{supp}}(M, \textbf{m}) = f(M, \textbf{m})$.
\vspace{6pt}

\noindent
2. In the \textit{non-admissible} case where the mobility is not positive semi-definite $0 \not\prec_{\mathcal{S}_{+}} M$, we consider the eigenvalue decomposition of $M = U^{T}\Lambda U$, where $\Lambda = \text{diag}[\lambda_{\alpha}]_{\alpha = 1}^{n}$. Since $0 \not\prec_{\mathcal{S}_{+}} M$, we can always find an negative eigenvalue $\lambda_{\alpha}$ of $M$, and we then consider
\begin{equation*}
    Q_{n} = U^{T}\text{diag}\Big(\underbrace{[0, \cdots, -n, \cdots, 0]}_{\alpha\text{-th is non-zero}}\Big) U,\quad \boldsymbol{q}_{n} = 0.
\end{equation*}
Hence, we have $M : Q_{n} + \textbf{m} : \boldsymbol{q}_{n} = -n\lambda_{\alpha}$, which converges to $\infty$ as $n \to \infty$.
\vspace{6pt}

\noindent
3. In the \textit{non-admissible} case where $\textbf{m}$ is non-vanishing in the kernel of the mobility $M$, i.e., $\textbf{m} \neq M\textbf{u}$ for all $\textbf{u}$, we consider the eigenvalue decomposition $M = U^{T}\Lambda U$, where $\Lambda = \text{diag}[\lambda_{\alpha}]_{\alpha = 1}^{n}$. Since $\textbf{m} \neq M\textbf{u}$ for all $\textbf{u}$, there exists a row $(U\textbf{m})_{\alpha} \neq 0$ such that $\lambda_{\alpha} = 0$, we then let
\begin{equation*}
    Q_{n} = U^{T}\text{diag}\Big(\underbrace{[0, \cdots, -\tfrac{n^{2}}{2}|(U\textbf{m})_{\alpha}|^{2}, \cdots, 0]}_{\alpha\text{-th is non-zero}}\Big) U,\quad \boldsymbol{q}_{n} = nU^{T}
    \left[\begin{array}{c}
        0\\
        \vdots\\
        (U\textbf{m})_{\alpha}\\
        \vdots\\
        0
    \end{array}\right].
\end{equation*}
Hence, we have $M : Q_{n} + \textbf{m} : \boldsymbol{q}_{n} = n|(U\textbf{m})_{\alpha}|^{2}$, which converges to $\infty$ as $n \to \infty$.
\end{proof}

\section{Conclusion}\label{sec:5}
In this work, we constructed a primal-dual forward-backward (PDFB) splitting method for computing cross-diffusion gradient flows via solving their minimizing movements at the fully discrete level. By using the PDFB splitting method developed in \cite{PDFB2024}, we are able to propose a computational framework that can efficiently reduce the minimizing movements to a general convex-concave saddle point problem, where the action functionals involving mobility matrices are decoupled via auxiliary dual variables.

There are several future directions associated with this work. One is to develop numerical solvers for general reaction-diffusion systems \cite{MielkeAlexander2011, LieroMatthias2013Gsag} based on the current PDFB splitting method. For the numerical analysis, it would be crucial to prove the convergence of the minimizing movements to the cross diffusion gradient flows for general dimensions, where one dimensional case have been addressed in \cite{ZinslJonathan2015Tdag}. Furthermore, it can be important to prove the discrete-to-continuous convergence of the staggered grid approximated distance to the “continuous" transport distance~\eqref{SSD}, where these results have been shown for the Wasserstein distance and distances induced by scalar mobilities \cite{gigli2013gromov, gladbach2020scaling, garcia2020gromov, gladbach2020homogenisation}. Besides, the authors believe that the current PDFB splitting methods for cross-diffusion systems can be potentially extended to computing optimal transport of matrix valued measures \cite{brenier2020optimal} which finds its application in quantum mechanics \cite{chen2019interpolation}, and dynamics of incompressible fluid \cite{brenier2018initial}.

\section*{Acknowledgments}
The authors would like to thank Jos\'e A. Carrillo and Li Wang for their valuable observations and suggestions, and Rafael Bailo for his helpful discussion on Example 4.

\bibliographystyle{siamplain}
%\bibliography{references}

\begin{thebibliography}{10}
\bibitem{Ambrosio2005}
{\sc L.~Ambrosio, N.~Gigli, and G.~Savar{\'e}}, {\em Gradient flows: in metric spaces and in the space of probability measures}, Springer, 2005.

\bibitem{Carrillo2022}
{\sc R.~Bailo, J.~A. Carrillo, and J.~Hu}, {\em Bound-preserving finite-volume schemes for systems of continuity equations with saturation}, SIAM J. Appl. Math., 83 (2023), pp.~1315--1339.

\bibitem{Barrett2003}
{\sc J.~W. Barrett, H.~Garcke, and R.~N{\"u}rnberg}, {\em Finite element approximation of surfactant spreading on a thin film}, SIAM J. Numer. Anal., 41 (2003), pp.~1427--1464.

\bibitem{Barrett2004}
{\sc J.~W. Barrett and R.~N{\"u}rnberg}, {\em Convergence of a finite-element approximation of surfactant spreading on a thin film in the presence of van der {Waals} forces}, IMA J. Numer. Anal., 24 (2004), pp.~323--363.

\bibitem{Benamou2000}
{\sc J.-D. Benamou and Y.~Brenier}, {\em A computational fluid mechanics solution to the {Monge-Kantorovich} mass transfer problem}, Numer. Math., 84 (2000), pp.~375--393.

\bibitem{Benamou2016}
{\sc J.-D. Benamou, G.~Carlier, and M.~Laborde}, {\em An augmented lagrangian approach to {Wasserstein} gradient flows and applications}, ESAIM Proc., 54 (2016), pp.~1--17.

\bibitem{Blanchet2008}
{\sc A.~Blanchet, V.~Calvez, and J.~A. Carrillo}, {\em Convergence of the mass-transport steepest descent scheme for the subcritical {Patlak-Keller-Segel} model}, SIAM J. Numer. Anal., 46 (2008), pp.~691--721.

\bibitem{Boffi2013}
{\sc D.~Boffi, F.~Brezzi, and M.~Fortin}, {\em Mixed Finite Element Methods and Applications}, Springer, 2013.

\bibitem{Boyd2004}
{\sc S.~Boyd and L.~Vandenberghe}, {\em Convex Optimization}, Cambridge University Press, 2004.

\bibitem{brenier2018initial}
{\sc Y.~Brenier}, {\em The initial value problem for the {Euler} equations of incompressible fluids viewed as a concave maximization problem}, Commun. Math. Phys., 364 (2018), pp.~579--605.

\bibitem{brenier2020optimal}
{\sc Y.~Brenier and D.~Vorotnikov}, {\em On optimal transport of matrix-valued measures}, SIAM J. Math. Anal., 52 (2020), pp.~2849--2873.

\bibitem{cances2019simulation}
{\sc C.~Canc{\`e}s, T.~Gallou{\"e}t, M.~Laborde, and L.~Monsaingeon}, {\em Simulation of multiphase porous media flows with minimising movement and finite volume schemes}, Eur. J. Math., 30 (2019), pp.~1123--1152.

\bibitem{cances2022convergent}
{\sc C.~Canc{\`e}s and A.~Zurek}, {\em A convergent finite volume scheme for dissipation driven models with volume filling constraint}, Numer. Math., 151 (2022), pp.~279--328.

\bibitem{Carrillo2019}
{\sc J.~A. Carrillo, K.~Craig, L.~Wang, and C.~Wei}, {\em Primal dual methods for {Wasserstein} gradient flows}, Foundations of Computational Mathematics, 22 (2021), pp.~389 -- 443.

\bibitem{Carrillo2003}
{\sc J.~A. Carrillo, R.~J. McCann, and C.~Villani}, {\em Kinetic equilibration rates for granular media and related equations: entropy dissipation and mass transportation estimates}, Rev. Mat. Iberoam., 19 (2003), pp.~971--1018.

\bibitem{Carrillo2023}
{\sc J.~A. Carrillo, L.~Wang, and C.~Wei}, {\em Structure preserving primal dual methods for gradient flows with nonlinear mobility transport distances}, SIAM J. Numer. Anal., 62 (2024), pp.~376--399.

\bibitem{Chambolle2011}
{\sc A.~Chambolle and T.~Pock}, {\em A first-order primal-dual algorithm for convex problems with applications to imaging}, J. Math. Imaging Vision, 40 (2011), pp.~120--145.

\bibitem{Chambolle2016}
{\sc A.~Chambolle and T.~Pock}, {\em On the ergodic convergence rates of a first-order primal–dual algorithm}, Math. Program., 159 (2016), pp.~253--287.

\bibitem{chen2019interpolation}
{\sc Y.~Chen, T.~T. Georgiou, and A.~Tannenbaum}, {\em Interpolation of matrices and matrix-valued densities: The unbalanced case}, Eur. J. Math., 30 (2019), pp.~458--480.

\bibitem{david2024degenerate}
{\sc N.~David, T.~D{\c{e}}biec, M.~Mandal, and M.~Schmidtchen}, {\em A degenerate cross-diffusion system as the inviscid limit of a nonlocal tissue growth model}, SIAM J. Math. Anal., 56 (2024), pp.~2090--2114.

\bibitem{PDFB2024}
{\sc Y.~Deng, L.~Wang, and C.~Wei}, {\em Efficient primal-dual forward-backward splitting method for {Wasserstein}-like gradient flows with general nonlinear mobilities}, arXiv preprint arXiv:2504.12713,  (2025).

\bibitem{DesvillettesLaurent2007}
{\sc L.~Desvillettes, K.~Fellner, M.~Pierre, and J.~Vovelle}, {\em Global existence for quadratic systems of reaction-diffusion}, Adv. Nonlinear Stud., 7 (2007), pp.~491--511.

\bibitem{ducasse2023cross}
{\sc R.~Ducasse, F.~Santambrogio, and H.~Yolda{\c{s}}}, {\em A cross-diffusion system obtained via (convex) relaxation in the {JKO} scheme}, Calc. Var. Partial Differential Equations, 62 (2023), p.~29.

\bibitem{Elliott2000}
{\sc C.~Elliott and H.~Garcke}, {\em Diffusional phase transitions in multicomponent systems with a concentration dependent mobility matrix}, Physica D: Nonlinear Phenomena, 109 (2000), pp.~242--256.

\bibitem{Fagioli2022}
{\sc S.~Fagioli and O.~Tse}, {\em On gradient flow and entropy solutions for nonlocal transport equations with nonlinear mobility}, Nonlinear Anal., 221 (2022), p.~112904.

\bibitem{Mass2022}
{\sc D.~Forkert, J.~Maas, and L.~Portinale}, {\em Evolutionary {$\Gamma$}-convergence of entropic gradient flow structures for {Fokker-Planck} equations in multiple dimensions}, SIAM J. Math. Anal., 54 (2022), pp.~4297--4333.

\bibitem{fu2024generalized}
{\sc G.~Fu, S.~Osher, W.~Pazner, and W.~Li}, {\em Generalized optimal transport and mean field control problems for reaction-diffusion systems with high-order finite element computation}, J. Comput. Phys., 508 (2024), p.~112994.

\bibitem{Li2023}
{\sc G.~Fu, S.~J. Osher, and W.~Li}, {\em High order spatial discretization for variational time implicit schemes: {Wasserstein} gradient flows and reaction-diffusion systems}, J. Comput. Phys., 491 (2023), p.~112375.

\bibitem{gao2023homogenization}
{\sc Y.~Gao and N.~K. Yip}, {\em Homogenization of {Wasserstein} gradient flows}, arXiv preprint arXiv:2312.01584,  (2023).

\bibitem{garcia2020gromov}
{\sc N.~Garc{\'\i}a~Trillos}, {\em {Gromov-Hausdorff} limit of {Wasserstein} spaces on point clouds}, Calc. Var. Partial Differential Equations, 59 (2020), p.~73.

\bibitem{Garcke2006}
{\sc H.~Garcke and S.~Wieland}, {\em Surfactant spreading on thin viscous films: Nonnegative solutions of a coupled degenerate system}, SIAM J. Math. Anal., 37 (2006), p.~2025–2048.

\bibitem{gigli2013gromov}
{\sc N.~Gigli and J.~Maas}, {\em {Gromov--Hausdorff} convergence of discrete transportation metrics}, SIAM J. Math. Anal., 45 (2013), pp.~879--899.

\bibitem{GiovangigliVincent2004}
{\sc V.~Giovangigli and M.~Massot}, {\em Entropic structure of multicomponent reactive flows with partial equilibrium reduced chemistry}, Math. Methods Appl. Sci., 27 (2004), pp.~739--768.

\bibitem{gladbach2020scaling}
{\sc P.~Gladbach, E.~Kopfer, and J.~Maas}, {\em Scaling limits of discrete optimal transport}, SIAM J. Math. Anal., 52 (2020), pp.~2759--2802.

\bibitem{gladbach2020homogenisation}
{\sc P.~Gladbach, E.~Kopfer, J.~Maas, and L.~Portinale}, {\em Homogenisation of one-dimensional discrete optimal transport}, Journal de Math{\'e}matiques Pures et Appliqu{\'e}es, 139 (2020), pp.~204--234.

\bibitem{Goldluecke2012}
{\sc B.~Goldluecke, E.~Strekalovskiy, and D.~Cremers}, {\em The natural vectorial total variation which arises from geometric measure theory}, SIAM J. Imaging Sci., 5 (2012), pp.~537--563.

\bibitem{Hamedani2018}
{\sc E.~Y. Hamedani and N.~S. Aybat}, {\em A primal-dual algorithm for general convex-concave saddle point problems}, arXiv preprint arXiv:1803.01401, 2 (2018), p.~71.

\bibitem{Jachalski2012}
{\sc S.~Jachalski, R.~Huth, G.~Kitavtsev, D.~Peschka, and B.~Wagner}, {\em Stationary solutions of liquid two-layer thin-film models}, SIAM J. Appl. Math., 73 (2012), pp.~1183--1202.

\bibitem{Jordan1996}
{\sc R.~Jordan, D.~Kinderlehrer, and F.~Otto}, {\em The variational formulation of the fokker--planck equation}, SIAM J. Math. Anal., 29 (1998), pp.~1--17.

\bibitem{kim2018nonlinear}
{\sc I.~Kim and A.~R. M{\'e}sz{\'a}ros}, {\em On nonlinear cross-diffusion systems: an optimal transport approach}, Calc. Var. Partial Differential Equations, 57 (2018), p.~79.

\bibitem{Li2019}
{\sc W.~Li, J.~Lu, and L.~Wang}, {\em Fisher information regularization schemes for wasserstein gradient flows}, J. Comput. Phys., 416 (2020), p.~109449.

\bibitem{LieroMatthias2013Gsag}
{\sc M.~Liero and A.~Mielke}, {\em Gradient structures and geodesic convexity for reaction-diffusion systems}, Philos. Trans. A, 371 (2013), pp.~20120346--20120346.

\bibitem{Savare2016}
{\sc M.~Liero, A.~Mielke, and G.~Savar{\'e}}, {\em Optimal transport in competition with reaction: The {Hellinger--Kantorovich} distance and geodesic curves}, SIAM J. Math. Anal., 48 (2016), pp.~2869--2911.

\bibitem{Lisini2012}
{\sc S.~Lisini, D.~Matthes, and G.~Savar'e}, {\em {Cahn–Hilliard} and thin film equations with nonlinear mobility as gradient flows in weighted-{Wasserstein} metrics}, J. Differential Equations, 253 (2012), pp.~814--850.

\bibitem{Malitsky2018}
{\sc Y.~Malitsky and M.~K. Tam}, {\em A forward-backward splitting method for monotone inclusions without cocoercivity}, SIAM J. Optim., 30 (2018), pp.~1451--1472.

\bibitem{Manzoni2021}
{\sc A.~Manzoni, A.~M. Quarteroni, and S.~Salsa}, {\em Optimal Control of Partial Differential Equations: Analysis, Approximation, and Applications}, Springer, 2021.

\bibitem{MielkeAlexander2011}
{\sc A.~Mielke}, {\em A gradient structure for reaction-diffusion systems and for energy-drift-diffusion systems}, Nonlinearity, 24 (2011), p.~1329.

\bibitem{murphy2019control}
{\sc T.~J. Murphy and N.~J. Walkington}, {\em Control volume approximation of degenerate two-phase porous flows}, SIAM J. Numer. Anal., 57 (2019), pp.~527--546.

\bibitem{Oron1997}
{\sc A.~Oron, S.~H. Davis, and S.~G. Bankoff}, {\em Long-scale evolution of thin liquid films}, Rev. Modern Phys., 69 (1997), pp.~931--980.

\bibitem{Otto2001}
{\sc F.~Otto}, {\em The geometry of dissipative evolution equations: The porous medium equation}, Commun. Partial Differ. Equ., 26 (2001), pp.~101--174.

\bibitem{Otto2005}
{\sc F.~Otto and M.~Westdickenberg}, {\em Eulerian calculus for the contraction in the {Wasserstein} distance}, SIAM J. Math. Anal., 37 (2005), pp.~1227--1255.

\bibitem{Papadakis2013}
{\sc N.~Papadakis, G.~Peyr{\'e}, and E.~Oudet}, {\em Optimal transport with proximal splitting}, SIAM J. Imaging Sci., 7 (2014), pp.~212--238.

\bibitem{Pock2011}
{\sc T.~Pock and A.~Chambolle}, {\em Diagonal preconditioning for first order primal-dual algorithms in convex optimization}, 2011 International Conference on Computer Vision,  (2011), pp.~1762--1769.

\bibitem{Walkington}
{\sc B.~Seguin and N.~Walkington}, {\em Multi-component multiphase porous flow}, Arch. Ration. Mech., 235 (2020), p.~2171–2196.

\bibitem{SKT}
{\sc N.~Shigesada, K.~Kawasaki, and E.~Teramoto}, {\em Spatial segregation of interacting species}, J. Theoret. Biol., 79 (1979), pp.~83--99.

\bibitem{Sun2018}
{\sc Z.~Sun, J.~A. Carrillo, and C.-W. Shu}, {\em An entropy stable high-order discontinuous {Galerkin} method for cross-diffusion gradient flow systems}, Kinet. Relat. Models,  (2018).

\bibitem{Thiele2012}
{\sc U.~Thiele, A.~J. Archer, and M.~Plapp}, {\em Thermodynamically consistent description of the hydrodynamics of free surfaces covered by insoluble surfactants of high concentration}, Physics of Fluids, 24 (2012).

\bibitem{Villani2003}
{\sc C.~Villani}, {\em Optimal Transport: Old and New}, Springer Berlin, 2009.

\bibitem{Yan2016}
{\sc M.~Yan}, {\em A new primal–dual algorithm for minimizing the sum of three functions with a linear operator}, J. Sci. Comput., 76 (2016), pp.~1698--1717.

\bibitem{Zhu2022}
{\sc Y.~Zhu, D.~Liu, and Q.~Tran-Dinh}, {\em New primal-dual algorithms for a class of nonsmooth and nonlinear convex-concave minimax problems}, SIAM J. Optim., 32 (2022), pp.~2580--2611.

\bibitem{ZinslJonathan2015Tdag}
{\sc J.~Zinsl and D.~Matthes}, {\em Transport distances and geodesic convexity for systems of degenerate diffusion equations}, Calc. Var. Partial Differential Equations, 54 (2015), pp.~3397--3438.
\end{thebibliography}

\end{document}